\documentclass[a4paper]{amsart}

\usepackage{amsfonts}
\usepackage{amssymb}
\usepackage{amsmath}
\usepackage{hyperref}
\usepackage{mathrsfs}
\usepackage{centernot}
\usepackage{mathdots}
\usepackage{stmaryrd}
\usepackage[all]{xy}

\usepackage{mathtools}

\usepackage{color}

\newtheorem{thm}{Theorem}[section]
\newtheorem{lem}[thm]{Lemma}
\newtheorem{prp}[thm]{Proposition}
\newtheorem{cor}[thm]{Corollary}
\newtheorem{dfn}[thm]{Definition}

\theoremstyle{definition}

\newtheorem{example}[thm]{Example}
\newtheorem{remark}[thm]{Remark}

\theoremstyle{plain}

\newcommand{\rem}[1]{}

\newcommand{\C}{\mathbb{C}}

\newcommand{\F}{\mathbb{F}}

\newcommand{\N}{\mathbb{N}}
\newcommand{\Q}{\mathbb{Q}}
\newcommand{\R}{\mathbb{R}}
\newcommand{\Z}{\mathbb{Z}}

\newcommand{\HH}{{\mathrm{H}}}

\newcommand{\frakp}{{\mathfrak{p}}}

\newcommand{\calM}{{\mathcal{M}}}

\newcommand{\calO}{{\mathcal{O}}}

\newcommand{\calU}{{\mathcal{U}}}

\newcommand{\bbA}{{\mathbb{A}}}

\newcommand{\bbP}{{\mathbb{P}}}

\newcommand{\catC}{{\mathscr{C}}}

\newcommand{\vphi}{\varphi}

\newcommand{\idealof}{\unlhd} 

\newcommand{\embeds}{\hookrightarrow}

\newcommand{\xonto}[2][]{%
  \xrightarrow[#1]{#2}\mathrel{\mkern-14mu}\rightarrow
}

\newcommand{\suchthat}{\,:\,}
\newcommand{\where}{\,|\,}

\newcommand{\quo}[1]{\overline{#1}}

\newcommand{\Trings}[1]{\left< #1 \right>}

\newcommand{\Floor}[1]{\left\lfloor {#1} \right\rfloor}

\DeclareMathOperator{\Char}{char} %
\DeclareMathOperator{\Ext}{Ext} %
\DeclareMathOperator{\ind}{ind} %
\DeclareMathOperator{\Int}{Int} %
\DeclareMathOperator{\Nr}{Nr} %
\DeclareMathOperator{\Pic}{Pic} %
\DeclareMathOperator{\rad}{rad} %
\DeclareMathOperator{\rank}{rank}
\DeclareMathOperator{\Span}{span} %
\DeclareMathOperator{\Spec}{Spec} %
\newcommand{\nGL}[2]{\mathrm{GL}_{#2}({#1})}

\newcommand{\uGL}{{\mathbf{GL}}}

\newcommand{\uSL}{{\mathbf{SL}}}

\newcommand{\uGa}{{\mathbf{G}}_{\mathbf{a}}}
\newcommand{\uGm}{{\mathbf{G}}_{\mathbf{m}}}

\newcommand{\nGm}[1]{{\mathbf{G}}_{\mathbf{m},{#1}}}
\newcommand{\nGa}[1]{{\mathbf{G}}_{\mathbf{a},{#1}}}

\newcommand{\umu}{{\boldsymbol{\mu}}}

\newcommand{\fppf}{\mathrm{fppf}}

\usepackage[foot]{amsaddr}

\usepackage{enumitem} 
\usepackage{xfrac}

\DeclareMathOperator{\CH}{CH} 

\DeclareMathOperator{\spl}{spl} 

\DeclareMathOperator{\gen}{gen} 

\DeclareMathOperator{\ed}{ed} 

\newcommand{\Sch}{{\mathsf{Sch}}}

\newcommand{\ualpha}{{\boldsymbol{\alpha}}}

\newcommand{\algcl}[1]{\overline{#1}} 

\title[Torsors That Are Versal for All Affine Varieties]{Algebraic Groups with Torsors That Are Versal for All Affine Varieties}

\author{Uriya A.\ First$^1$}
\address{$^1$University of Haifa}
\email{uriya.first@gmail.com}

\author{Mathieu Florence$^2$}
\address{$^2$Université Paris Cité and Sorbonne Université, CNRS, IMJ-PRG, Paris}
\email{mflorence@imj-prg.fr}

\author{Zev Rosengarten$^3$}
\address{$^3$Hebrew University}
\email{zev.rosengarten@mail.huji.ac.il}

\begin{document}

\maketitle


\begin{abstract} 
Let $k$ be a field and let $G$ be an affine algebraic group over $k$.
Call a  $G$-torsor   \emph{weakly versal} for a class of $k$-schemes $\catC$
if it specializes to every $G$-torsor over a scheme in $\catC$.
A recent result of the first author, Reichstein and Williams says that for any $d\geq 0$, there exists
a $G$-torsor  over a finite type $k$-scheme that is weakly versal for   finite type affine $k$-schemes
of dimension at most $d$.
The first author also observed that if $G$ is unipotent,
then $G$ admits a torsor over a finite type $k$-scheme that is weakly versal for \emph{all} affine $k$-schemes,
and that the converse holds if $\Char k=0$. In this work, we extend this  to all fields, showing
that $G$ is unipotent if and only if it admits a $G$-torsor over a quasi-compact base that is weakly versal for
all  finite type regular affine $k$-schemes.
Our proof is   characteristic-free  and it also gives rise to a  quantitative statement:
If $G$ is a non-unipotent subgroup of $\uGL_n$, then a $G$-torsor over a quasi-projective $k$-scheme
of dimension $d$ is \emph{not} weakly versal for finite type regular affine $k$-schemes of dimension $n(d+1)+2$. 
This means in particular that every
such $G$
admits a nontrivial  torsor over a  regular affine $(n+2)$-dimensional variety.
When $G$ contains a nontrivial torus, we  show that nontrivial torsors
already exist over $3$-dimensional smooth affine varieties (even when $G$ is \emph{special}), 
and this is optimal in general.

In the course of the proof, we show that for every $m,\ell\in\N\cup\{0\}$ with $\ell\neq 1$,
there exists a smooth affine $k$-scheme $X$ carrying an $\ell$-torsion line bundle that cannot be
generated by $m$ global sections. We moreover study the minimal possible dimension of such an $X$
and show that it is $m$, $m+1$ or $m+2$.
\end{abstract}

\section{Introduction}

Let $k$ be a field and let $G$ be an affine algebraic group over $k$.
Recall that a $G$-torsor $E\to X$, where $X$ is a $k$-scheme,
is said to be \emph{weakly versal} if $E\to X$ specializes to every
$G$-torsor over any $k$-field $K$. 
When   the joint image of
all such specializations $\Spec K\to X$ is   dense in $X$,
the torsor $E\to X$ is called \emph{versal};
see \cite{Duncan_2015_versality_of_alg_grp_actions}, for instance.
Versal torsors have many applications, notably
to study of essential dimension and the theory of cohomological invariants, e.g., see
\cite{Merkurjev_2013_ess_dim_survey} and \cite{Serre_2003_cohomological_invariants}. 
In
particular, the  essential dimension  of $G$, denoted $\ed_k(G)$,
is equal to the minimal possible dimension of a   $k$-variety that is the base of a versal
$G$-torsor.

Given a class of $k$-schemes $\catC$, let us
say that a $G$-torsor $E\to X$ is \emph{weakly versal for $\catC$}
if every $G$-torsor over a scheme in $\catC$ is a specialization of $E\to X$.
This definition first appeared in \cite[\S5]{First_2022_generators_of_alg_over_comm_ring},
where it was shown that if $k$ is infinite,
then for every $d\in \N$, there exists a $G$-torsor 
over a finite type $k$-scheme
that is weakly versal for the class of   finite type affine $k$-schemes of dimension at most
$d$.
Special examples of this phenomena were observed   earlier, e.g.,
\cite{Saltman_1982_generic_Galois_extensions} and
\cite[\S5.1]{Auel_2019_Azumaya_algebras_without_inv}. Weak versality for general 
classes of $k$-schemes, as well as    stronger variants of this property, were studied systematically by the first
author in \cite{First_2023_highly_versal_torsors}, and we refer to introduction of this source for a more
detailed history. 
By \cite[Thm.~8.2]{First_2023_highly_versal_torsors},
for every $d\in\N$ and every algebraic subgroup $G$ of $\uGL_n$, there is a $G$-torsor
over a quasi-projective
$k$-scheme of dimension $nd+n^2-\dim G$ 
that is    versal (in  the strongest  sense considered in \cite{First_2023_highly_versal_torsors}) for the class of  affine noetherian  
$k$-schemes of dimension at most $d$.

Perhaps more surprisingly, it was observed in 
\cite[Thm.~10.1]{First_2023_highly_versal_torsors}
that if $G$ is unipotent (see \cite[6.45, 14.5]{Milne_2017_algebraic_groups} for the definition), then there are $G$-torsors over   quasi-projective $k$-schemes
that are weakly versal for \emph{all} affine $k$-schemes.\footnote{
	Note that in characteristic $0$,
	every $G$-torsor over an affine scheme is trivial when $G$ is unipotent, but in positive
	characteristic, unipotent groups may have nontrivial torsors.}
The special case where $p:=\Char k>0$   and
$G$ is a finite constant $p$-group  goes back to Saltman \cite{Saltman_1982_generic_Galois_extensions}; see also \cite{Fleischmann_2011_non_linear_actions_of_p_groups}.
The first author also showed that the converse holds in characteristic $0$,
that is, an algebraic $k$-group admitting a $G$-torsor over a finite type $k$-scheme
that is versal for all affine $k$-schemes is unipotent
\cite[Thm.~10.6]{First_2023_highly_versal_torsors}. They asked whether
this also holds in positive characteristic. 
Our first main result answers this question in the affirmative.

\begin{thm}\label{TH:main}
	Let $G$ be an affine algebraic group over a field $k$.
	Then the following conditions are equivalent:
	\begin{enumerate}[label=(\alph*)]
		\item there exists a $G$-torsor over a quasi-compact $k$-scheme
		that is weakly versal for all finite type  regular affine $k$-schemes;
		\item there exists a $G$-torsor over a smooth quasi-projective $k$-scheme
		that is weakly versal for all   affine $k$-schemes;
		\item $G$ is unipotent.
	\end{enumerate}
\end{thm}

Our proof is very different from the argument in \cite{First_2023_highly_versal_torsors}.
It works in any characteristic, applies to torsors over a quasi-compact base (rather than
of finite type over $k$), and it also yields  a quantitative variant
for $G$-torsors over quasi-projective $k$-schemes, which is our second main result.

\begin{thm}\label{TH:main-quantitative}
	Let $k$ be a field, and let $G$ be a non-unipotent algebraic subgroup of the algebraic $k$-group $\uGL_n$. 
	Let $E\to X$ be a $G$-torsor such that $X$ is quasi-projective over a $k$-field
	and  
	of dimension $d$. Then $E\to X$ is \emph{not} weakly versal
	for the class of finite type regular affine $k$-schemes of dimension at most $n(d+1)+2$.
\end{thm}

Both theorems admit mild improvements when some assumptions are imposed on $G$ or $k$;
see Section~\ref{sec:main-proof}.
For example, if $G$ is connected  or $k$ is perfect, then we may replace ``finite type regular''
in condition (a) of Theorem~\ref{TH:main} and in Theorem~\ref{TH:main-quantitative}
with ``smooth''.
Furthermore, if $G$ contains a nontrivial torus, then
we can   assert in Theorem~\ref{TH:main-quantitative} that $E\to X$
is not versal for  smooth    affine $k$-schemes of dimension $\leq  n(d+1)+1$.
The strongest statements we can make about particular $G$ and $k$
are given in Theorems~\ref{TH:explicit-torsor} and~\ref{TH:explicit-torsor-torus}.

\medskip

Theorem~\ref{TH:main-quantitative}
implies that every non-unipotent 
algebraic subgroup of $\uGL_n$
admits a nontrivial $G$-torsor
over a finite type regular affine $k$-scheme
of dimension $\leq n+2$; mild improvements
of this are given in Corollaries~\ref{CR:there-are-nontriv-torsors}
and~\ref{CR:nontrivial-torsor-after-ext}.
However, we can do significantly better when $G$ contains a nontrivial torus.

\begin{thm}\label{TH:non-triv-torsors-main}
	Let $G$ be an affine algebraic $k$-group containing a nontrivial torus (resp.\ central torus).
	Then there exists a smooth affine $k$-scheme $X$ of dimension $\leq 3$ (resp.\ $1$) carrying
	a nontrivial $G$-torsor. When $\Char k=0$ and $\mathrm{trdeg}_{\Q}k$ is infinite,
	we can take $X$ to be of  dimension $\leq 2$.
\end{thm}

Theorem~\ref{TH:non-triv-torsors-main} is the best result
we can hope for with general $k$ and $G$.
Indeed, we show  that
any $\uSL_n$-torsor ($n>1$) over  a noetherian affine scheme of dimension smaller than $2$
is trivial (Example~\ref{EX:nontrivial-torsor-lower-bound}),
and 
moreover, every $\uSL_n$-torsor over a smooth affine $\quo{\F}_p$-scheme
of dimension smaller then $3$ is trivial (Corollary~\ref{CR:triviality-of-tors-dim-2}). 
The characteristic-$0$ analogue of the latter,
i.e., that every
$\uSL_n$-torsor over a smooth affine $\algcl{\Q}$-scheme  of dimension $\leq 2$
is trivial,
is equivalent to the Bloch--Beilinson conjecture for surfaces (Remark~\ref{RM:triviality-of-tors-dim-2}).

The ``local'' counterpart of Theorem~\ref{TH:non-triv-torsors-main}, i.e.,
the problem of finding nontrivial torsors over $k$-fields,  was already 
considered
by Reichstein and Youssin \cite{Reichstein_2001_property_special_groups}. 
To put their result in context, recall that Serre made two conjectures
about the triviality of torsors over fields of small cohomological dimension
\cite{Serre_1962_Galois_coh_alg_grps},
\cite[III.\S3]{Serre_2002_Galois_cohomology_english}.
Serre's Conjecture I --- now a theorem of Steinberg 
\cite{Steinberg_1965_regular_elements_ss_grps} --- states
that if $G$ is a connected algebraic group over a field $K$ of cohomological dimension $\leq 1$, then every $G$-torsor over $K$ is trivial.
Serre's Conjecture II 
states that if $G$ is a   semisimple\footnote{
	We follow the convention that reductive groups, and in particular semisimple groups, 
	are connected.
} simply connected 
algebraic group over a perfect field $K$ of cohomological dimension $\leq 2$,
then every $G$-torsor over $K$ is trivial; this conjecture is known in many cases
\cite{Gille_2010_Serre_conj_II}.
Serre further called an algebraic group $G$ over $k$ \emph{special} 
if every $G$-torsor over a $k$-field is trivial \cite{Serre_1958_espaces_fibres_algebriques}. The main result
of  \cite{Reichstein_2001_property_special_groups}, which we state here only
in characteristic $0$ for simplicity, informally
says
that Serre's conjectures I and II are optimal.
Formally, suppose that $G$ is an algebraic group over an algebraically closed field $k$ (with $\Char k = 0$). Then:
\begin{enumerate}[label=(\roman*)]
	\item If $G$ is not connected, then $G$ has nontrivial
	torsors over  any $k$-field of transcendence degree $1$.
	\item If $G/\mathrm{rad} (G)$ is not semisimple simply connected,
	then $G$ has nontrivial
	torsors over  any $k$-field of transcendence degree $2$.
	\item If $G$ is not special, then $G$ has nontrivial
	torsors over  any $k$-field of transcendence degree $3$. 
\end{enumerate}
By spreading out, these statements
imply a result similar to Theorem~\ref{TH:non-triv-torsors-main},
but 
which applies only to non-special groups. 
In addition to addressing the case of special groups containing a nontrivial torus,
Theorem~\ref{TH:non-triv-torsors-main}   improves the  ``global''  consequences
of (i)--(iii) when
$G$ is  semisimple simply connected  and $\mathrm{trdeg}_{\Q}k$ is infinite by
showing that there exist nontrivial $G$-torsors
over \emph{$2$-dimensional} smooth affine $k$-schemes. (There are also additional
improvements in   positive characteristic.)
On the other hand, the global analogue
of (i) gives a better result than Theorem~\ref{TH:non-triv-torsors-main}
for non-connected (smooth) algebraic groups.

\medskip

We return to the case where $G$ is an affine algebraic group over some field $k$.
Another way to look at Theorem~\ref{TH:main-quantitative} is via the following corollary:

\begin{cor}\label{CR:essential-dim-bound}
	Let $G$ be an algebraic subgroup of the algebraic $k$-group $\uGL_n$,
	and let $m_d(G)$ be the smallest  $m\in\N\cup\{0\}$
	for which there is a $G$-torsor    over an $m$-dimensional
	quasi-projective $k$-scheme that is weakly versal for finite type affine $k$-schemes
	of dimension at most $  d$. If $G$ is not unipotent, then
	\[
	\Floor{\frac{d-2}{n}}   \leq m_d(G)\leq nd+n^2-\dim G.
	\]
	Otherwise,  
	\[
	m_d(G)\leq \frac{n(n-1)}{2}-\dim G.
	\]
\end{cor}

\begin{proof}
	The lower bound follows from Theorem~\ref{TH:main-quantitative}. 
	The upper bounds follow from \cite[Thms.~8.2, 10.1]{First_2023_highly_versal_torsors}
\end{proof}

When $d=0$, the integer $m_d(G)$ is just the essential dimension of $G$,\footnote{
	This is true despite the fact
	that Corollary~\ref{CR:essential-dim-bound} considers only   torsors 
	over quasi-projective $k$-schemes. Indeed, we clearly
	have $\ed_k(G)\leq m_0(G)$. Furthermore, if $p:E\to X$
	is a versal $G$-torsor, then there is a versal $G$-torsor
	$E'\to X'$ with $\dim X'=\dim X$ and $X'$ affine (and in particular quasi-projective)  
	--- simply take a dense open affine
	$X'\subseteq X$ and $E'=p^{-1}(X')$. Thus, $m_0(G)$ is smaller than $\min\{\dim X\where
	\text{$E\to X$ is a versal $G$-torsor}\}$,
	which is $\ed_k(G)$.
} so 
the invariants $\{m_d(G)\}_{d\geq 0}$ may be regarded as ``higher'' analogues of the essential dimension.
Corollary~\ref{CR:essential-dim-bound} tells us that
the sequence $\{m_d(G)\}_{d\geq 0}$ is either bounded --- which is the case precisely
when $G$ is unipotent ---
or it has linear growth rate.

\medskip

We prove Theorems~\ref{TH:main} and~\ref{TH:main-quantitative}
by establishing three results. 

First, we show in Theorem~\ref{unipiffnomultsubgp} that an affine algebraic $k$-group 
$G$ is not unipotent if and only if $G_{\algcl{k}}$ 
contains a copy of $\umu_\ell$ (the group of $\ell$-th roots of unity)
for some $\ell>1$; this was known
if $G$ is smooth and connected \cite[Ch.\,IV, Cor.\,11.5(2)]{Borel_1991_Linear_Algebraic_Groups}. We
also show that a general  connected algebraic  $k$-group
(possibly not affine or not smooth) that is not unipotent
has a nontrivial multiplicative algebraic subgroup.

Next, we show in Theorem~\ref{TH:reduction-to-line-bundles} that if $G$ is linear and 
contains a copy of $\umu_\ell$,
and if there is a $G$-torsor over a quasi-compact scheme $X$ that is weakly versal
for a class $\catC$ of affine $k$-schemes, then there is a uniform bound
on the number of generators of every $\ell$-torsion line bundle over a scheme in $\catC$.
This upper bound is explicit if $X$ is quasi-projective over  a $k$-field.

Finally, we show in 
Theorem~\ref{TH:bounds-on-generators}
that for every $m,\ell\in\N\cup \{0\}$ with $\ell\neq 1$,
there exists a smooth affine $k$-scheme $X$ 
admitting an $\ell$-torsion line bundle that cannot be generated by $m$ global sections;
such examples were known for $\ell=2$  if $\Char k\neq 2$ \cite[Prop.~7.13, Prop.~6.3]{Shukla_2021_classifying_spaces_for_etale_algebras}.
More generally, we study what is the minimal possible dimension of $X$ (also in the non-torsion
case $\ell=0$) and show that it is $m$, $m+1$ or $m+2$.
A key results which we establish along the way
is the following criterion  (Lemma~\ref{LM:power-of-c1}):
If the $m$-th power of the $1$-st Chern class of a line bundle $L$ is nonzero,
then $L$
cannot be generated by
$m$ global sections. 

Putting Theorems~\ref{unipiffnomultsubgp},
\ref{TH:reduction-to-line-bundles} and~\ref{TH:bounds-on-generators}
together tells us (after a little work) that if $G$ is not unipotent, then
no  $G$-torsor over a quasi-compact base
$E\to X$ is weakly versal for all finite type regular affine $k$-schemes. This gives Theorem~\ref{TH:main}.
When $X$ is quasi-projective, 
we can moreover point to a particular $G$-torsor $E'\to X'$ that is not a specialization of $E\to X$ 
and thus derive the quantitative 
Theorem~\ref{TH:main-quantitative}.

Theorem~\ref{TH:non-triv-torsors-main} is not a corollary
of Theorem~\ref{TH:main-quantitative}, but its proof follows a simpler
version of the above  strategy.
Loosely speaking, we make use of the torus contained in $G$ rather than the copy of $\umu_\ell$.

\medskip

The paper is organized as follows. Section~\ref{sec:prelim} recalls necessary
facts about algebraic groups, torsors and vector bundles. 
Sections~\ref{sec:nonunip-grps} and \ref{sec:reduction}
are dedicated to proving Theorems~\ref{unipiffnomultsubgp} and \ref{TH:reduction-to-line-bundles},
respectively.
In Section~\ref{sec:torsion-line-bundles}, we give many examples
of both torsion and non-torsion line bundles over smooth
affine schemes requiring many global sections to generate,
thus proving Theorem~\ref{TH:bounds-on-generators}.
These results are used in Section~\ref{sec:main-proof} to prove 
Theorems~\ref{TH:main} and~\ref{TH:main-quantitative}.
Finally, Theorem~\ref{TH:non-triv-torsors-main} is proved in Section~\ref{sec:nontriv}.

\subsection*{Acknowledgements}

We are grateful to Z.\ Reichstein for comments and suggestions.
We also thank 
V.\ Srinivas for a correspondence concerning Section~\ref{sec:torsion-line-bundles}. Finally, we thank the anonymous referees for their comments.

The first named author is supported by ISF grant no.\ 721/24.
The third named author is supported by ISF grant no.\ 2083/24.

\section{Preliminaries}
\label{sec:prelim}

Throughout this work, $k$ denotes a field and $\algcl{k}$ is an algebraic closure of
$k$. If not indicated otherwise,
all schemes and group schemes are over $k$,   morphisms are $k$-morphisms, and products
of schemes are taken over $\Spec k$. By a $k$-variety we mean an integral separated  $k$-scheme
of finite type.
The $n$-dimensional affine and projective spaces
over $k$ are denoted $\bbA^n$ and $\bbP^n$, respectively.

A $k$-ring means a commutative (unital) $k$-algebra. 
A $k$-field is a field extension of $k$. If $A$ is a $k$-ring and $X$ is a $k$-scheme,
we write $X_A$ or $X\times_k A$ for $X\times_{\Spec k}\Spec A$   viewed as an $A$-scheme.
For example, $\bbP^n_A$ is the $n$-dimensional projective space over $\Spec A$.
Similar notation applies to $k$-morphisms.

A $k$-group is a group scheme over $k$.
An algebraic $k$-group or an algebraic group over $k$
means a $k$-group of finite type (possibly non-smooth).
Recall that an algebraic $k$-group  $G$ is affine if and only if it is linear.
The identity connected component of $G$ is denoted $G^\circ$ and its unipotent radical
is $\rad_u(G)$. By a \emph{multiplicative} algebraic group, we mean an algebraic group of multiplicative type.

As usual,
$\uGa$ is the additive group of $k$,
$\uGm$ is the multiplicative group of $k$, 
$\umu_\ell$ is the algebraic $k$-group of $\ell$-th roots of unity,
$\uGL_n$ is the algebraic $k$-group of invertible $n\times n$ matrices
and $\uSL_n$ is the algebraic $k$-group of $n\times n$ matrices of determinant $1$.
When $p=\Char k>0$, we also write $\ualpha_{p^r}$ for the subgroup of $\uGa$
cut by the equation $x^{p^r}=0$.
In accordance with our earlier conventions, 
given a $k$-field $K$, we  write $\nGa{K}$ for the algebraic $K$-group
$\uGa\times_k K$, and likewise for $\uGm$, $\umu_\ell$, etcetera.

We shall sometimes view $k$-schemes as sheaves on the site  of all $k$-schemes with the
fppf topology, denoted $(\Sch/k)_\fppf$. In particular, $k$-morphisms $f:X\to Y$
will sometimes be defined by specifying their action on sections, namely,
by specifying $f_S:X(S)\to Y(S)$ for every scheme $S$. 
When $G$ is a $k$-group  acting on a $k$-scheme $X$,
we will write $X/G$ to denote the sheafification of the presheaf $S\mapsto X(S)/G(S)$
on $(\Sch/k)_\fppf$. We shall treat $X/G$ as a $k$-scheme when it is represented by one.
For example, if $G$ is a subgroup of an algebraic  $k$-group $H$,
then $H/G$ is known to be a $k$-scheme, which is moreover quasi-projective when
$H$ is affine \cite[Thm.~5.28, \S7e, proof of Thm.~7.18]{Milne_2017_algebraic_groups}.

\medskip

Let $G$ be a $k$-group scheme. Our conventions regarding $G$-torsors are the same
as in \cite[\S1]{First_2023_highly_versal_torsors}. That is,
a $G$-torsor  over a $k$-scheme $X$ consists of a sheaf
$E$ on $(\Sch/k)_\fppf$ together with a morphism
$E\to X$ and a \emph{right} $G$-action $E\times G\to E$
subject to the requirement that there is an fppf covering $X'\to X$
for which $X'\times_X E \cong X'\times_k G$ as schemes over $X'$
carrying a $G$-action.
Thus, $G$-torsors over $X$ are classified by the cohomology pointed set
$\HH^1_\fppf(X,G)$.
In general, not all $G$-torsors are represented by a $k$-scheme,
but this is true if $G$ is affine \cite[III, Thm.~4.3]{Milne_1980_etale_cohomology}.

Let $E\to X$ be a $G$-torsor and let $f:Y\to X$ be a morphism.
Then the first projection $ Y\times_X E\to Y$ is also a $G$-torsor;
we write $f^*E$ for $Y\times_X E$. We say that a $G$-torsor
$E\to X$ \emph{specializes} to another $G$-torsor $E'\to X'$
if there is a $k$-morphism $f:X'\to X$ such that $E'\cong f^*E$ as $G$-torsors
over $Y$. This is equivalent to the existence of a $G$-equivariant map $\hat{f}:E'\to E$.

Given a morphism of $k$-groups $\vphi:G\to H$
and a $G$-torsor
$p:E\to X$, one can form the $H$-torsor 
$$\vphi_*(p):E\times^GH\to X.$$ 
Here,
$E\times^G H$ is the quotient
of $E\times_k H$ by the equivalence relation $$(e,g,h)\mapsto ((eg,h),(e,\vphi(g)h)):
E\times  G\times  H\to (E\times  H)\times  (E\times H),$$
and $\vphi_*(p)$ is given by $\vphi_*(p)[e,h]=p(e)$ on sections.

\medskip

A vector bundle over a $k$-scheme $X$ is a locally free $\calO_X$-module $V$.
In this case, the functor $(f:Y\to X)\mapsto \Gamma(Y,f^*V)$ from $X$-schemes to sets
is represented by a scheme which we also denote by $V$.  (Here,
as usual,   $f^*V=f^{-1}V\otimes_{f^{-1}\calO_X} \calO_Y$.)

Recall that there is a bijection between isomorphism classes 
of rank-$n$ vector bundles on $X$ and isomorphism classes of $\uGL_n$-torsors over $X$. Moreover,
this equivalence is compatible with base-change.
Briefly, given a rank-$n$ vector bundle $V$, we can choose a Zariski covering $\{U_i\to X\}_{i\in I}$
and   isomorphisms $\{\vphi_i:\calO_{U_i}^n\to V_{U_i}\}_{i\in I}$. Setting $U_{ij}=U_i\cap  U_j$, 
this data gives rise to a $\uGL_n$-valued
$1$-cocycle on $X$, namely  $\alpha=\{\vphi_j^{-1}|_{U_{ij}}\circ\vphi_i|_{U_{ij}}\}_{i,j\in I}$,
and it defines the $\uGL_n$-torsor corresponding to $V$ (as well as its cohomology
class in $\HH^1_\fppf(X,\uGL_n)$).

A line bundle $L$ over $X$ is a rank-$1$ vector bundle over $X$. We say that $L$
is \emph{$\ell$-torsion} ($\ell\in \N\cup\{0\}$) if $L^{\otimes\ell}:=L\otimes_{\calO_X}\dots\otimes_{\calO_X} L$
($\ell$ times)
is isomorphic to $\calO_X$. 
Note that every line bundle is $0$-torsion.
By the previous paragraph, there is a one-to-one correspondence
between isomorphism classes of line bundles over $X$, 
isomorphism classes of $\uGm$-torsors over $X$  and $\HH^1_\fppf(X,\uGm)$.
For $\ell>0$, the Kummer short exact sequence $\umu_\ell \embeds \uGm\xonto{\ell} \uGm$  of sheaves on $(\Sch/k)_\fppf$
gives rise to a long cohomology exact sequence
\[
\dots \to \HH^1_\fppf(X,\umu_\ell)\to \HH^1_\fppf(X,\uGm)\xrightarrow{\ell}\HH^1_\fppf(X,\uGm)\to\dots ,
\]
from which we see that every $\umu_\ell$-torsor gives rise to an $\ell$-torsion line bundle,
and every $\ell$-torsion line bundle arises from some $\umu_\ell$-torsor (that is not unique
in general).
Explicitly, the line bundle arising from a $\umu_\ell$-torsor $T\to X$ corresponds
to the $\uGm$-torsor $T\times^{\umu_\ell} \uGm\to X$.

\section{A Characterization of Non-Unipotent Groups}
\label{sec:nonunip-grps}

Recall the following important result: Among \emph{smooth} connected linear algebraic groups over a field $k$,  unipotent groups are characterized by the property that they do not contain a nontrivial multiplicative subgroup. When $k$ is algebraically closed, this is \cite[Ch.\,IV, Cor.\,11.5(2)]{Borel_1991_Linear_Algebraic_Groups}, and over general fields it follows from Grothendieck's theorem on the existence of geometrically maximal tori (that is, tori defined over $k$ which are maximal over $\overline{k}$) \cite[Exp.\,XIV, Th.\,1.1]{SGAiii}. Our goal in this section is to prove an analogous result for all connected algebraic groups over general fields, and arbitrary algebraic groups over algebraically closed fields. The result for connected affine groups over algebraically closed fields is contained in SGA \cite[Exp.\,XVII, Prop.\,4.3.1, (i)$\iff$(iv)]{SGAiii}. In fact, the same proof works over perfect fields. The purpose of this section is to prove the same result without affineness or perfectness hypotheses.

We begin with the following lemma.

\begin{lem}
\label{extmultunip}
Over a field $k$, an extension of a nontrivial finite connected multiplicative algebraic $k$-group $M$ by a finite connected unipotent algebraic $k$-group $U$ contains a nontrivial multiplicative algebraic $k$-group.
\end{lem}

\begin{proof}
We may assume that ${\rm{char}}(k) = p > 0$ (otherwise
$M$ is trivial), and that  $M$ is $p$-torsion  
\cite[Cor.~11.30]{Milne_2017_algebraic_groups}. 
By \cite[Exp.\,XVII, Prop.\,4.2.1,(i)$\iff$(iii)]{SGAiii}, we may also assume that $U = \ualpha_p^r$. The proof is now essentially identical to that given in \cite[Cor.\,15.33]{Milne_2017_algebraic_groups}, using \cite[Prop.\,15.31]{Milne_2017_algebraic_groups}.
\end{proof}

The key case, when $G$ is finite connected, is contained in the following lemma.

\begin{lem}
\label{unip=nomultsubgp}
Let $k$ be a field of characteristic $p>0$. A finite connected $k$-group scheme $I$ is unipotent if and only if it does not contain a nontrivial multiplicative subgroup.
\end{lem}

\begin{proof}
The ``only if'' direction is trivial, so we concentrate on the converse. First assume that $I$ is commutative. The connected-\'etale sequence of $\widehat{I}$, the Cartier dual of $I$, then shows that $\widehat{I}$ must be connected (since $I$ does not contain a multiplicative group), hence $I$ is a commutative infinitesimal (that is, finite connected) $k$-group with infinitesimal dual. This implies that it is unipotent. (In fact, for this it is not hard to show that it is enough to have infinitesimal dual.) Indeed, by filtering we may suppose that $I$ and its dual are both killed by Frobenius, and Dieudonn\'e theory \cite[Th.\,1.4.3.2]{cco} then implies that, at least over the perfection of $k$, $I$ is isomorphic to a power of $\ualpha_p$, which then implies that it is unipotent; see \cite[Exp.\,XVII, Def.\,1.3]{SGAiii}.

We now prove the result in general by induction on the order of $I$. Note first that if one has an exact sequence
\[
1 \longrightarrow I' \longrightarrow I \longrightarrow I'' \longrightarrow 1
\]
with $I' \trianglelefteq I$ proper and nontrivial, then by induction $I'$ is unipotent. If $I''$ contained a nontrivial multiplicative subgroup scheme, then so would $I$ by Lemma \ref{extmultunip}. Therefore, $I''$ does not contain such a copy, so is also unipotent by induction, hence so is $I$. We are therefore free to assume that $I$ does not admit a normal nontrivial proper $k$-subgroup. In particular, we may assume that the Frobenius morphism of $I$ is trivial. That is, $I$ is a height one algebraic group. 

Thanks to the equivalence between height one  algebraic $k$-groups and $p$-Lie algebras \cite[Exp.\,VII$_A$, Thm.\,8.1.2]{SGAiii}, and the fact that we have already handled the commutative case, the problem translates into the following $p$-Lie algebra question: If $\mathfrak{g}$ is a nonabelian $p$-Lie algebra admitting no proper nonzero $p$-ideals (that is, ideals preserved by the $p$ operation), then $\mathfrak{g}$ contains a nontrivial abelian $p$-Lie subalgebra with surjective $p$ operation. Indeed, this follows from the fact that multiplicative groups have surjective $p$ operation (just check over an algebraic closure where it reduces to the case of $\umu_n$) while unipotent commutative groups of height one have nilpotent $p$ operation (as all such groups over an algebraic closure admit a filtration by either $\ualpha_p$ or $\Z/p\Z$), plus the fact that all commutative finite $k$-groups may be filtered by multiplicative and unipotent ones.

First assume that ${\rm{ad}}(X)$ is nilpotent for all $X \in \mathfrak{g}$. Then Engel's theorem implies that $\mathfrak{g}$ is a nilpotent Lie algebra. In particular, it has nontrivial center. Since ${\rm{ad}}(X^{[p]}) = {\rm{ad}}(X)^p$, the center is preserved by the $p$ operation. It follows from our assumption that $\mathfrak{g}$ has no proper nonzero ideals that $\mathfrak{g}$ is abelian, contrary to assumption.

Now assume that not all ${\rm{ad}}(X)$ are nilpotent. Let $X \in \mathfrak{g}$ be such that ${\rm{ad}}(X)$ is not nilpotent. If $f(T) \in k[T]$ is the minimal polynomial of ${\rm{ad}}(X)$, then $f$ is not a power of $T$. Furthermore, for some $r \geq 0$, $f(T) = g(T^{p^r})$ with $g \in k[T]$ separable. Because ${\rm{ad}}(X^{[p^r]}) = {\rm{ad}}(X)^{p^r}$, $g$ is the minimal polynomial of ${\rm{ad}}(X^{[p^r]})$. Therefore, replacing $X$ by $X^{[p^r]}$,  we may assume that ${\rm{ad}}(X)$ is non-nilpotent and that its minimal polynomial is separable over $k$. It follows that $X$ admits a Jordan-Chevalley decomposition $X = X_s + X_n$ with ${\rm{ad}}(X_s)$ geometrically semisimple, and ${\rm{ad}}(X_n)$ nilpotent. Replacing $X$ by $X_s$, therefore, we obtain nonzero $X \in \mathfrak{g}$ such that ${\rm{ad}}(X)$ is geometrically semisimple.

Consider the nonzero subspace $\mathfrak{h} := \langle X, X^{[p]}, X^{[p^2]}, \dots \rangle \subset \mathfrak{g}$. We claim that $\mathfrak{h}$ is abelian. Indeed, this follows from the fact that, for any $Y \in \mathfrak{g}$, ${\rm{ad}}(Y^{[p^n]}) = {\rm{ad}}(Y)^{p^n}$, hence $Y^{[p^n]}$ commutes with $Y$. This equation also shows that ${\rm{ad}}(Y) \in \mathfrak{gl}(\mathfrak{g})$ is geometrically semisimple for all $Y \in \mathfrak{h}$. Because $\mathfrak{h}$ is abelian, one has $(Y+Z)^{[p]} = Y^{[p]} + Z^{[p]}$ for any $Y, Z \in \mathfrak{h}$. Thus $\mathfrak{h} \subset \mathfrak{g}$ is a (nonzero) abelian $p$-Lie subalgebra of $\mathfrak{g}$ all of whose elements have geometrically semisimple adjoint action on $\mathfrak{g}$, and furthermore this property is preserved upon passage to $\overline{k}$.

We must show that the $p$ operation on $\mathfrak{h}$ is surjective, and it suffices to show this after extending scalars to $\algcl{k}$, where it is equivalent to injectivity, so we may assume that $k$ is algebraically closed. Suppose we are given $Y \in \mathfrak{h}$ such that $Y^{[p]} = 0$. We must show that $Y = 0$. One has ${\rm{ad}}(Y^{[p]}) = {\rm{ad}}(Y)^p = 0 \in \mathfrak{gl}(\mathfrak{g})$. Because ${\rm{ad}}(Y)$ is (geometrically) semisimple, it follows that ${\rm{ad}}(Y) = 0$. That is, $Y$ is central in $\mathfrak{g}$. Since $Y^{[p]} = 0$, it follows that $\langle Y\rangle$ is a $p$-ideal of $\mathfrak{g}$, a contradiction because we assumed that $\mathfrak{g}$ is nonabelian with no proper nonzero $p$-ideals. This contradiction shows that the $p$ operation on $\mathfrak{h}$ is injective, and completes the proof of the lemma.
\end{proof}

\begin{lem}
\label{extmultunipperf}
Let $k$ be a perfect field. Then any extension $E$ of a multiplicative \'etale algebraic $k$-group $M$ by a connected unipotent algebraic $k$-group $U$ splits.
\end{lem}

\begin{proof}
By \cite[Exp.\,XVII, Prop.\,4.3.1(i)$\iff$(iii)]{SGAiii}, we may assume that $U$ is a power of either $\mathbf{G}_a$ or $\ualpha_p$. The former case follows from \cite[Prop.\,15.31]{Milne_2017_algebraic_groups}, while the latter follows from \cite[Exp.\,XVII, Lem.\,1.6]{SGAiii}.
\end{proof}

The following two lemmas will be required in order to deal with groups that are not affine.
The non-affine case will not be needed in the next sections.

\begin{lem}
\label{nonaffabvarquot}
If $G$ is a connected algebraic group over a field $k$, then there is an exact sequence of 
algebraic groups over $k$
\[
1 \longrightarrow H \longrightarrow G \longrightarrow A \longrightarrow 1
\]
with $H$ affine and connected and $A$ an abelian variety.
\end{lem}

\begin{proof}
By \cite[Exp.\,VII$_{\rm{A}}$, Prop.\,8.3]{SGAiii}, there is an infinitesimal normal algebraic $k$-group $I \trianglelefteq G$ such that $G/I$ is smooth, so we are free to assume that $G$ is smooth. If $k$ is perfect, then we are done by Chevalley's Theorem. Otherwise, $G_{k_{{\rm{perf}}}}$ may be written as an extension of an abelian variety by a linear algebraic $k$-group, and therefore $G_{k^{p^{-n}}}$ may be so written for some $n \geq 0$. Via the isomorphism $k^{p^{-n}} \xrightarrow{\sim} k$, $\lambda \mapsto \lambda^{p^n}$, $G_{k^{p^{-n}}}$ is identified with $G^{(p^n)}$, the $n$-fold Frobenius twist of $G$. Therefore, $G^{(p^n)}$ may be written as an extension of an abelian variety by a linear algebraic group. Since $G$ is smooth, the $n$-fold Frobenius map $G \rightarrow G^{(p^n)}$ is surjective, hence the result for $G$.
\end{proof}

\begin{lem}
\label{extabvarunipsogsplit}
Over a field $k$ of characteristic $p > 0$, for any extension $E$ of an abelian variety $A$ by a unipotent algebraic $k$-group $U$, there is an integer $n \geq 0$ such that the extension splits upon pullback by the map $[p^n]\colon A \rightarrow A$.
\end{lem}

\begin{proof}
We may assume that $U$ is either connected or finite \'etale. In the latter case, we may by filtration assume that $U$ is commutative and $p$-torsion. In the former case, \cite[Exp.\,XVII, Prop.\,4.3.1(i)$\iff$(iii)]{SGAiii} allows us to assume that $U$ is a power of either $\ualpha_p$ or $\uGa$. Thus we may assume that $U$ is either $p$-torsion finite \'etale commutative, or else a power of either $\ualpha_p$ or $\uGa$. If the extension $E$ is commutative, then because $U$ is $p$-torsion, $\Ext^1(A, U)$ is killed by $[p]$, hence $[p]_A$ induces the $0$ map on $\Ext^1(A, U)$, which proves the lemma. It therefore only remains to prove that $E$ is commutative.

First suppose that $U$ is either finite \'etale or a power of $\ualpha_p$. Then the automorphism functor ${\rm{Aut}}_{U/k}$ of $U$ is an affine $k$-group scheme, and the conjugation action of $A$ on $U$ is then given by a homomorphism $A \rightarrow {\rm{Aut}}_{U/k}$, which must be trivial. Therefore, $E$ is a central extension. The commutator map of $E$ therefore descends to a map $A \times A \rightarrow U$, which must be constant because $U$ is finite and $A \times A$ is connected and geometrically reduced. Thus $E$ is commutative.

It remains to show that $E$ is commutative under the assumption that $U = \uGa^r$ for some $r > 0$. Once again, conjugation induces an action of $A$ on $\uGa^r$. But we claim that any such action is trivial, and hence $E$ is central. To prove this, we may temporarily extend scalars and thereby assume that $k$ is infinite. Then for any $x \in k^r = \uGa^r(k)$, the map $a \mapsto a\cdot x$ induces a map $A \rightarrow \uGa^r$. Because $\uGa^r$ is affine, this map must be constant, hence $A$ acts trivially on $x$. Because $\uGa^r(k)$ is Zariski dense in $\uGa^r$, it follows that the action of $A$ on $\uGa^r$ is trivial and $E$ is central, as claimed. Once again, therefore, the commutator map of $E$ descends to a map $A \times A \rightarrow \uGa^r$ which must be constant because the target is affine. Thus $E$ is commutative, and the proof is complete.
\end{proof}

We now prove the main result of this section.

\begin{thm}
\label{unipiffnomultsubgp}
Let $k$ be a field  and $G$ an algebraic group over $k$.
\begin{itemize}
\item[(i)] If $G$ is connected, then $G$ is unipotent if and only if it does not contain a nontrivial multiplicative algebraic  $k$-group.
\item[(ii)] If $k$ is algebraically closed, then $G$ is unipotent if and only if it does not contain a nontrivial multiplicative algebraic  $k$-group.
\end{itemize}
\end{thm}

\begin{proof}
The ``only if'' direction is immediate, so we concentrate on the ``if'' direction. First we prove (i). Let $G$ be a connected algebraic group over $k$ that has no nontrivial multiplicative subgroup. We must show that $G$ is unipotent. Let us first give the proof under the additional assumption that $G$ is affine. By \cite[Exp.\,VII$_{\rm{A}}$, Prop.\,8.3]{SGAiii}, there is an infinitesimal (i.e., finite connected) algebraic subgroup $I \trianglelefteq G$ such that $\overline{G} := G/I$ is smooth (and connected), and $I$ is unipotent by Lemma~\ref{unip=nomultsubgp}. If $\overline{G}$ contains a nontrivial $k$-torus, then Lemma \ref{extmultunip} implies that $G$ contains a nontrivial multiplicative subgroup. Therefore, $\overline{G}$ does not contain a nontrivial torus, hence is unipotent by \cite[Ch.\,IV, Cor.\,11.5(2)]{Borel_1991_Linear_Algebraic_Groups} and \cite[Exp.\,XIV, Thm.\,1.1]{SGAiii}, which completes the proof in the affine case.

In the general (not necessarily affine) case, Lemma \ref{nonaffabvarquot} furnishes an exact sequence
\[
1 \longrightarrow H \longrightarrow G \longrightarrow A \longrightarrow 1
\]
with $H$ connected and affine, and $A$ an abelian variety. By the already-treated affine case, $H$ is unipotent. Assume for the sake of contradiction that $A \neq 0$. First suppose that $\Char k = 0$, or even just that $k$ is perfect. For any $n > 1$ not divisible by $\Char k$, $A[n]$ is a nontrivial \'etale multiplicative algebraic group
\cite[\href{https://stacks.math.columbia.edu/tag/03RP}{Tag 03RP}]{stacks_project}, hence Lemma \ref{extmultunipperf} implies that $G$ contains such an algebraic group, in violation of our assumption.

Now suppose that $\Char k = p > 0$. By Lemma \ref{extabvarunipsogsplit}, $G$ contains a copy of $A$. If $A$ were nontrivial, it would follow that $G$ contains a nontrivial multiplicative subgroup --- for instance, $A[n]$ for any $n > 1$ not divisible by $\Char k$. It follows that $A = 0$.

For (ii), let $G$ be an algebraic group over $k$ not containing a nontrivial multiplicative subgroup. By (i), $G^0$ is unipotent, and we need only show that the finite constant algebraic $k$-group $E := G/G^0$ is unipotent -- that is, $E(k)$ is a $p$-group. If not, then $E(k)$ admits a nontrivial $n$-torsion element for some $n$ not divisible by $p$. This element then defines a nontrivial multiplicative \'etale $k$-subgroup of $E$, so $G$ contains such a group by Lemma \ref{extmultunipperf}. This violates our assumption about $G$, hence completes the proof.
\end{proof}

\begin{remark}
The assumption of connectedness in Theorem~\ref{unipiffnomultsubgp}(i) is necessary beyond the algebraically closed case. Indeed, for $k$ imperfect of characteristic $p$, for instance, \cite[Cor.\,15.33]{Milne_2017_algebraic_groups} implies that there are non-split extensions of $\umu_{\ell}$ by $\ualpha_p$ for primes $\ell \neq p$. Such an extension is not unipotent, but it cannot contain a nontrivial multiplicative subgroup, as this would yield a splitting of the extension.
\end{remark}

\section{Reduction to a Statement About Torsion Line Bundles}
\label{sec:reduction}

Let $k$ be a field and let $G$ be a linear algebraic 
group over $k$.
The purpose of this section is to show that if a $G$-torsor $E\to X$
is weakly versal for a class of affine $k$-schemes $\catC$ and $G$ contains a copy of $\umu_\ell$
(resp.\ $\uGm$), then
there is a  uniform upper bound on the number of   generators of an $\ell$-torsion (resp.\ any)
line-bundle over a scheme
in $\catC$. This is stated formally in Theorem~\ref{TH:reduction-to-line-bundles}
which makes  use of the following invariant.

\begin{dfn}
    Let $V$ be a vector bundle of   rank $n$ 
    over a scheme $X$. An open covering  $\calU$ 
    of $X$ is said to split $V$ if $V_{U}\cong\calO_{U}^n$ for all $U\in \calU$.
    The smallest possible cardinality of a covering which splits $V$ is called
    the  \emph{splitting number} of $V$ and denoted
    \[
    \spl(V).
    \]
\end{dfn}

Observe that $\spl(V)$ is finite if $X$ is quasi-compact.
Furthermore, for every morphism $f:Y\to X$,
we have $\spl(f^*V)\leq \spl(V)$.

\begin{prp}\label{PR:bound-on-splitting-number}
    Let $X$ be a noetherian scheme such that every finite collection of closed points of $X$
    is contained in an open affine subscheme of $X$, e.g., a quasi-projective scheme over a field. 
    Let $V$ be a rank-$n$ vector bundle over $X$. Then $\spl(V)\leq \dim X+1$.
\end{prp}

\begin{proof}
	It is enough to show that there is a sequence $\{U_i\}_{i=1}^{\dim X+1}$ of open affine subsets of
	$X$ such that $V_{U_r}\cong \calO_{U_r}^n$ and $\dim (X- \bigcup_{i=1}^r U_i)\leq \dim X-r$ for all $r\in\{0,1,\dots,\dim X+1\}$. Indeed, if this holds, then $\{U_i\}_{i=1}^{\dim X+1}$
	is a covering of $X$ which splits $V$. 
	We prove the existence of the sequence by induction. Let $r\in\{1,\dots,\dim X+1\}$
	and suppose that $U_1,\dots,U_{r-1}$ were constructed.
	Put $Y=X- \bigcup_{i=1}^{r-1} U_i $
	and let $y_1,\dots,y_t$ be closed points of $Y$ meeting every irreducible component of $Y$.
	These points are also closed in $X$, so there is an open affine $U\subseteq X$
	with $y_1,\dots,y_t\in U$. Since $U$ is affine,  the semilocalization of $V_U$ at $\{y_1,\dots,y_t\}$
	is free. We may therefore shrink $U$ to assume that $V_U\cong \calO_U^n$.
	Now take $U_r$ to be $U$. We have $V_{U_r}\cong \calO_{U_r}^n$
	by construction, and since $U_r$ meets every irreducible component of $Y$,
	we also have $\dim (X- \bigcup_{i=1}^{r} U_i)=\dim (Y-U_r)\leq \dim Y-1\leq \dim X-(r-1)-1=\dim X-r$.
\end{proof}

Let $X$ be a scheme, let $\calM$ be an $\calO_X$-module and let $S\subseteq \Gamma(X,\calM)$.
As usual, we say that $S$  generates $\calM$ if the $\calO_X$-module map
\[(\alpha_s)_{s\in S}\mapsto \sum_{s\in S}\alpha_s s:\bigoplus_S \calO_X\to \calM\] 
is surjective.
The $\calO_X$-module $\calM$ 
is said to be \emph{globally generated} if there is some $S\subseteq \Gamma(X,\calM)$
generating $\calM$, and in this case we write
\[
\gen (\calM)
\]
for the minimal possible cardinality of such $S$.
This number is always finite if $X$ is affine and $\calM$ is quasicoherent   of finite type.

\begin{prp}\label{PR:generators-vs-splitting-number}
	Let $X$ be an affine scheme and let $V$ be a rank-$n$ vector bundle over $X$.
	Then $\gen (V)\leq n\spl(V)$.
\end{prp}

\begin{proof}
	Write $r=\spl(V)$. Then $X$ has an open covering $\{U_1,\dots,U_r\}$
	splitting $V$. 
	Write $X=\Spec A$. Then there are ideals $I_1,\dots,I_r\idealof A$
	such that
	\[U_i=\{\frakp \in\Spec A\suchthat I_i\nsubseteq \frakp\}\] for every $i\in\{1,\dots,r\}$.
	Since $X=\bigcup_{i=1}^r U_i$, we have $\sum_{i=1}^r I_i=A$.
	Thus, we can choose $a_i\in I_i$ for every $i$
	such that $\sum_{i=1}^r a_i=1_A$. Let $W_i$ be the principal open affine corresponding
	to $a_i$, namely,
	\[
	W_i=\{\frakp\in \Spec A\suchthat a_i\notin \frakp\}.
	\]
	Then $X=\bigcup_{i=1}^r W_i$ and $W_i\subseteq U_i$ for all $i$.
	
	Let $i\in \{1,\dots,r\}$.
	Since $V_{U_i}\cong \calO_{U_i}^n$, we also have $V_{W_i}\cong\calO_{W_i}^n$. Thus, there exist $f_1^{(i)},\dots,f_n^{(i)}\in \Gamma(W_i,V)$
	that generate $V_{W_i}$. By construction,  
	$\Gamma(W_i,V)$ is the localization
	of the $A$-module $\Gamma(X,V)$ at the multiplicative set $\{1,a_i,a_i^2,\dots\}$,
	so there are $g_1^{(i)},\dots,g_n^{(i)}\in \Gamma(X,V)$
	and $k\in\N$
	such that $f_j^{(i)}=a_i^{-k} g_j^{(i)}$ for all $j\in\{1,\dots,n\}$.
	
	Finally, put $S=\{g^{(i)}_j\where i\in\{1,\dots,r\},j\in\{1,\dots,n\}\}$. Then
	every $x\in X$ admits an open neighborhood where $S$ generates $V$. Since generating $V$ is a Zariski-local property, this means that $S$ generates $V$, and  it follows that
	$\gen(V)\leq |S|\leq nr$.
\end{proof}

We now prove  a key lemma.

\begin{lem}\label{LM:pre-reduction}
	Let $k$ be a field and let $\ell$
	be a prime number.
	Let $M$ be   $\uGm$ or $\umu_\ell$, 
	let $\rho:M\to \uGL_n$ be a monomorphism of algebraic $k$-groups,
	and let $\{x\mapsto x^{a_i}\}_{i=1}^n$ denote the characters of $M$
	occurring in the representation $\rho$. 
	Let $E\to X$ be a $\uGL_n$-torsor corresponding to a rank-$n$
	vector bundle $V$, and let $F\to Y$ be an $M$-torsor inducing a line bundle $L$. Suppose further that $Y$ is affine and $E\to X$
	specializes to $F\times^{M}\uGL_n\to Y$.  
	Then:
	\begin{enumerate}
		\item[(i)] When $M=\umu_\ell$, we have $\gen(L)\leq n\spl(V)$.
		\item[(ii)] When $M=\uGm$, we have $\gen(L)\leq (n\spl(V))^r$,
		where $r$ is the minimum of $|S|$ as $S$ ranges over all nonempty subsets of
		$\{a_1,\dots,a_n\}$ with $\gcd(S)=1$.
	\end{enumerate}
\end{lem}

\begin{proof}
	We prove (i) and (ii) together.
	Write $T=F\times^M\uGL_n$ 
	and denote by $W$ the rank-$n$ vector bundle over 
	$Y$ corresponding to $T$.

	Since $M$ is diagonalizable,	
	we may replace $\rho$ with $\Int(g)\circ \rho$
	for some $g\in\nGL{k}{n}$  to assume that $\rho$ is given section-wise by
	\[
	\vphi(x)=\begin{bmatrix}
	x^{a_1} & & \\
	& \ddots & \\
	& & x^{a_n}
	\end{bmatrix}.
	\]
	This does not affect the isomorphism class of the $\uGL_n$-torsor  $T\to Y$. 
	
	By assumption, there is a $k$-morphism $f:Y\to X$ such that $f^*E\cong T$ as $\uGL_n$-torsors. As a result, $f^*V\cong W$ and so
	$
	\spl(W)\leq \spl(V)
	$. Since $Y$ is affine, we may apply
	Proposition~\ref{PR:generators-vs-splitting-number} to $W$ and get
	\[
	\gen(W)\leq n\spl(V).
	\]
	
	Next, observe that since
	$T=M\times^{\umu_\ell} \uGL_n$,
	we have  
	\[W\cong L^{\otimes a_1}\oplus\dots \oplus L^{\otimes a_n},\]
	where here, $L^{\otimes a}$ means $(L^\vee)^{\otimes(-a)}$ if $a<0$.
	This can be seen   by inspecting the  cocycles representing
	the classes in $\HH^1_{\fppf}(Y,\uGL_n)$
	corresponding to both vector bundles. It follows that 
	each 
	$L^{\otimes a_i}$  is an epimorphic
	image of $W$, and therefore  $\gen(L^{\otimes a_i})\leq \gen(W)\leq n\spl(V)$.
	By a lemma of Swan  \cite[Lem.~3.4]{Bloch_1989_zero_cycles_num_generators}, we moreover have $\gen((L^{\otimes a_{i }})^{\otimes b })
	\leq \gen(L^{\otimes a_i})\leq n\spl(V)$ for all $b\in \Z$.
	
	Suppose now that $H=\umu_\ell$. Since $\rho$ is a monomorphism,
	there is some $1\leq i\leq n$ with $\ell\nmid a_i$.
	Thus, there is $b\in \Z$ such that $a_ib\equiv 1\, (\operatorname{mod} \ell)$.
	For that $b$, we get $\gen(L)=\gen((L^{\otimes a_i})^{\otimes b})\leq n\spl(V)$, as required.
	
	Finally, suppose that $H=\uGm$. Since $\rho$ is a monomorphism,
	$\gcd(a_1,\dots,a_n)=1$.
	Choose a subset $S\subseteq \{a_1,\dots,a_n\}$ of cardinality $r$
	such that $\gcd(S)=1$, say $S=\{a_{i_1},\dots,a_{i_r}\}$. Then 
	there are $b_1,\dots,b_r\in \Z$ such that  $\sum_{j=1}^r a_{i_j}b_j=1$.
	Then $L\cong \bigotimes_{j=1}^r (L^{\otimes a_{i_j}})^{\otimes b_j}$,
	and by what we have shown, 
	$\gen(L)\leq \prod_{i=1}^r\gen((L^{\otimes a_{i_j}})^{\otimes b_j})\leq (n\spl(V))^r$.
\end{proof}

\begin{thm}\label{TH:reduction-to-line-bundles}
	Let $k$ be a field and let $G$ be a closed subgroup of $\uGL_n$ ($n\in\N$).
	Let $E\to X$ be a $G$-torsor,
	and let $V$ denote the 	rank-$n$ vector bundle corresponding to $E\times^G\uGL_n$.
	Suppose that $E\to X$ is versal for a class of affine $k$-schemes
	$\catC$.
	\begin{enumerate}
		\item[(i)] If $G$ contains a copy of $\umu_\ell$
		for some prime number $\ell$,
		then for every $Y\in \catC$ and every $\ell$-torsion line bundle
		$L$ on $Y$, we have $\gen(L)\leq n\spl(V)$.
		\item[(ii)] If $G$ contains a copy of $\uGm$,
		then for every $Y\in\catC$ and every line bundle
		$L$ on $Y$, we have $\gen(L)\leq (n\spl(V))^r$,
		where $r$ is the number  in Lemma~\ref{LM:pre-reduction}(ii) attached to
		the monomorphism $\uGm\to G\to \uGL_n$.
	\end{enumerate}
\end{thm}

\begin{proof}
	(i) Let $Y\in\catC$ and let $L$ be an $\ell$-torsion line bundle
	on $Y$.
	Choose a $\umu_\ell$-torsor $F\to Y$ inducing $Y$.
	Since the $G$-torsor $E\to X$ is weakly versal for $\catC$,
	it specializes to $F\times^{\umu_\ell}G$. This means that
	the $\uGL_n$-torsor
	$F\times^{\umu_\ell}\uGL_n=(F\times^{\umu_\ell}G)^{\times G}\uGL_n$
	(over $Y$)
	is a specialization of $E\times^{G} \uGL_n$ (over $X$).
	Applying Lemma~\ref{LM:pre-reduction}(i) to $E\times^{G} \uGL_n\to X$
	now gives the desired conclusion $\gen(L)\leq n\spl(V)$.
	
	(ii) The proof is the same as in (i) except we use
	Lemma~\ref{LM:pre-reduction}(ii).
\end{proof}

\section{Torsion Line Bundles Requiring Many Sections to Generate}
\label{sec:torsion-line-bundles}

Let $k$ be a field.
The purpose of this section is to show that for every $m\in\N$ and $1<\ell\in\N$, there
is an $\ell$-torsion line bundle  over a smooth affine $k$-scheme $X$ that requires
at least $m$ global sections to generate. This will be coupled with
Theorem~\ref{TH:reduction-to-line-bundles}
in the next section in order to prove Theorem~\ref{TH:main}.

For any $m\in\N$, the existence of  affine schemes of finite type over $\R$ admitting 
$2$-torsion line bundles which cannot be generated by $m$ elements goes back at
least as far as
Swan \cite[Thm.~4]{Swan_1962_vec_bundles_and_proj_modules}, who attributes the example to Chase.
Similar $2$-torsion 
examples over affine varieties
over any field of characteristic not $2$ were
given by Shukla and Williams
\cite[Props.~7.13, 6.3]{Shukla_2021_classifying_spaces_for_etale_algebras}. 

\medskip

Since the  dimension of the base scheme $X$ would affect the   quantitative version
of Theorem~\ref{TH:main}, i.e.\ 
Theorem~\ref{TH:main-quantitative}, we would like $\dim X$ to be as small as possible.
To that end, given $\ell,m\in\N\cup\{0\}$ with $\ell\neq 1$,
it is convenient to define
\[
g_k(\ell,m)
\]
to be the smallest $d\in \N\cup\{0\}$
for which there exists a $d$-dimensional smooth affine $k$-scheme
$X$ carrying an $\ell$-torsion line bundle $L$
satisfying
$\gen (L)>m$. Note that $g_k(0,m)$ is the minimal dimension of
a smooth affine $k$-scheme carrying a (possibly non-torsion)
line bundle which cannot be generated by $m$ elements.
By the end of this section, we will show:

\begin{thm}\label{TH:bounds-on-generators}
	Let $k$ be a field and let $\ell,m $ be non-negative integers with $\ell\neq 1$. 
	\begin{enumerate}[label=(\roman*)]
		\item  $g_k(\ell,0)=0$ and $g_k(\ell,1)=1$.
		\item  $g_k(\ell,m) \in\{m,m+1,  m+2\}$.  
		\item  If   $\ell$ is divisible by a prime number greater than
		$m$ (e.g.\ $\ell=0$),
		then   $g_k(\ell,m)\in\{m, m+1\}$.
		\item  If $k$ is algebraically closed, $m>1$ and $\ell>0$, 
		then $g_k(\ell,m)\geq m+1$. Equality holds
		if  $\ell$ is moreover divisible
		by a prime number greater
		than $m$.
		\item  If $\Char k=0$
		and $\mathrm{trdeg}_\Q k$ is infinite,
		then    $g_k(0,m)=m$.
		\item $g_{\overline \F_p }(0,m)=m+1$ for every prime $p>0$ and $m\geq 2$.
	\end{enumerate}
\end{thm}

We expect that part (v) also holds
in positive characteristic, see Remark~\ref{RM:positive-char}.
The characteristic-$0$ counterpart of (vi), i.e.,
the statement $g_{\algcl{\Q}}(0,m)=m+1$, is related to the   Bloch--Beilinson conjecture,
and in particular follows from it (Remark~\ref{RM:Bloch-Beilinson}).
We do not know if there are $k$, $\ell$, $m$
with $g_k(\ell,m)=m+2$; this is related to conjectures of Paulsen and Griffiths--Harris
(Remark~\ref{RM:Paulsen}).

In general, for $\ell>1$,
the precise  value of   $g_k(\ell,m)$  is sensitive to   
arithmetic properties of the field $k$. This demonstrated  in
Example~\ref{EX:Gk-and-u-inv} and Corollary~\ref{CR:upper-bound-on-Gk}.

\medskip

We begin by recalling a result of Forster \cite{Forster_1964_number_of_generators} (see
	\cite{First_2017_number_of_generators} for a more modern treatment): Every rank-$n$ projective module over a noetherian
	ring  of Krull dimension $m$ is generated
	by $m+n$ elements. This implies:

\begin{lem}\label{LM:Forster-bound}
	For every field $k$ and $\ell,m\in\N\cup\{0\}$ with $\ell\neq 1$, we have $g_k(\ell,m)\geq m$.
\end{lem}

In what follows,
given a smooth  (irreducible)  $k$-variety $X$, we denote by
$\CH^iX$ the $i$-th Chow group of $X$; see \cite[\S8.3]{Fulton_1998_intersection_theroy_second_ed}.
The Chow ring of $X$ is $\CH^*X = \bigoplus_{i\geq 0}\CH^iX$.
Recall that a  morphism $f:X\to Y$ between smooth   $k$-varieties
induces a ring homomorphism $f^*:\CH^*Y\to \CH^*X$.
If $E$ is a vector bundle over $X$, then the $i$-th Chern class of $E$
is denoted $c_i(E)\in\CH^iX$. The total Chern class of $E$
is $c(E)=\sum_{i= 0}^{\rank E} c_i(E)\in \CH^*X$; see \cite[\S3.2]{Fulton_1998_intersection_theroy_second_ed}
for details. 

\begin{prp}[Murthy {\cite[Cor.~3.16]{Murthy_1994_zero_cycles_and_proj_modules}}]\label{PR:Murthy}
	Suppose that $k$ is algebraically closed, let $X$ be a smooth affine $k$-variety
	of dimension $m$ and let $L$ be  a line bundle over $X$.
	Then $\gen(L)\leq m$ if and only if $c_1(L)^m=0$.
\end{prp}

It turns out that the ``only if'' part of Murthy's result holds for any
smooth scheme over any field; this will be our main  tool for bounding $g_k(\ell,m)$ from above.

\begin{lem}\label{LM:power-of-c1}
	Let $X$ be a smooth $k$-variety,   let $L$ be a line bundle over $X$ 
	and let $m\in\N$.
	Suppose that $c_1(L)^m\neq 0$ in $\CH^mX$. Then $\gen(L)>m$.
\end{lem}

\begin{proof}
	For the sake of contradiction, suppose that $L$ can be generated
	by $m$ global sections. Then there is an exact sequence of vector bundles on $X$,
	\[
	0\to E\to \calO_X^{\oplus m}\to L\to 0.
	\]
	By tensoring with $L^\vee$, the dual of $L$, we get an exact sequence
	\[
	0\to E'\to (L^\vee)^{\oplus m} \to \calO_X\to 0,
	\]
	where $E'=E\otimes L^\vee$. Thus,
	\[c(L^\vee)^m=c((L^\vee)^{\oplus m})=c(E')c(\calO_X)=c(E').\] 
	Since $\rank E'=m-1$,
	it follows that
	\[
	(-1)^mc_1(L)^m=c_1(L^\vee)^m = c_m(E')=0
	\]
	in $\CH^m X$. But this contradicts our assumptions.
\end{proof}  

In light of Proposition~\ref{PR:Murthy} and Lemma~\ref{LM:power-of-c1},
it is convenient to define
\[
G_k(\ell,m)
\]
to be the minimum possible dimension of a smooth affine $k$-scheme $X$
admitting an $\ell$-torsion line bundle $L$ satisfying $c_1(L)^m\neq 0$.
It is clear that $G_k(\ell,m)\geq m$,
and
Proposition~\ref{PR:Murthy} and Lemma~\ref{LM:power-of-c1}   imply: 

\begin{cor}\label{CR:gk-leq-Gk}
	For every field $k$ and integers $\ell,m\geq 0$ with $\ell\neq 1$, 
	we have  
	\[g_k(\ell,m)\leq G_k(\ell,m) .\]
	Moreover, when $k$ is algebraically closed,
	$g_k(\ell,m)=m$ if and only if $G_k(\ell,m)=m$.
\end{cor}

We can use the second assertion of Corollary~\ref{CR:gk-leq-Gk} and another result of Murthy to get a better lower
bound on $g_k(\ell,m)$.

\begin{prp}\label{PR:Gk-alg-closed-lower-bound}
	If $k$ is algebraically closed, $\ell>0$ and $m\geq 2$, then $G_k(\ell,m)\geq m+1$.
	Consequently, $g_k(\ell,m)\geq m+1$.
\end{prp}

\begin{proof}
	Let $L$ be an $\ell$-torsion line bundle over a smooth affine
	$m$-dimensional $k$-scheme $X$. Murthy \cite[Thms.~2.11, 2.14]{Murthy_1994_zero_cycles_and_proj_modules} showed that $\CH^m X$ is torsion-free. Since $c_1(L)$ is $\ell$-torsion,
	we must have $c_1(L)^m=0$.
\end{proof}

In the next results, we  establish upper bounds
on $G_k(\ell,m)$ 
by   exhibiting smooth affine $k$-varieties $X$ carrying an $\ell$-torsion line bundle $L$
such that $c_1(L)^m\neq 0$.
We will use the following proposition
to produce such examples.
Recall that the index of a $k$-scheme $X$, denoted $\ind X$, is the 
greatest common divisor of $\{\dim_k K\where \text{$K$ is a finite $k$-field
with $X(K)\neq\emptyset$}\}$.  

\begin{prp}\label{PR:Gk-bound-I}
	Let $m,\ell\in\N$  and
	let $Z$ be an integral hypersurface in $\bbP^m$ having degree $d$
	and index $n$.
	Let $X=\bbP^m-Z$, let
	$L=\calO_{\bbP^m}(1)|_X$,
	let $X'$ be the complement of the zero section in
	the total space of $\calO_{\bbP^m}(\ell)|_X$, let
	$p:X'\to X$ be its structure morphism
	and let $L'=p^*L$. Then:
	\begin{enumerate}
		\item[(i)] $X$ is affine, $L^{\otimes d}\cong \calO_X$  and the order
		of $c_1(L)^m$ in the (additive) group $\CH^m X$ is $n$.
		Consequently, $c_1(L)^m\neq 0$ if $n>1$.
		\item[(ii)] $X'$ is affine, $L'^{\otimes \ell}\cong \calO_{X'}$ 
		and the order of $c_1(L')^m$ in $\CH^m X'$ is $\gcd(n,\ell)$.
		Consequently, $c_1(L')^m\neq 0$ if $\gcd(n,\ell)>1$.
	\end{enumerate}
\end{prp}

\begin{proof}
	(i) The scheme $X$ is affine because it is the complement of an ample divisor
	in $\bbP^m$.
	
	Recall \cite[\S8.4]{Fulton_1998_intersection_theroy_second_ed}
	that $\CH^*\bbP^m=\Z[h\where h^{m+1}=0]$, where $h$ is the class of a hyperplane
	in $\bbP^m$ and lives in degree $1$.
	Write $\CH_i(Z)$ for the group of $i$-cycles on $Z$ modulo rational equivalence
	(we have $\CH_i(Z)=\CH^{m-1-i}Z$ when $Z$ is smooth). 
	By \cite[Prop.~1.8]{Fulton_1998_intersection_theroy_second_ed},
	there is an exact
	sequence 
	\begin{equation}\label{EQ:Chow-hypersurface-comp}
	\CH_i(Z)\to \CH^{m-i} \bbP^m\to \CH^{m-i} X\to 0
	\end{equation}
	in which the first map is pushforward along   $Z\to \bbP^m$
	and the second map is pullback along  $X\to \bbP^m$.
	Let $A_i$ denote the image  of $\CH_i(Z)$
	in $\CH^{m-i}\bbP^m=\Z h^{m-i}$.
	We may and shall identify
	$\CH^{m-i} X$ with $\Z h^{m-i}/A_i$ using \eqref{EQ:Chow-hypersurface-comp}.
	For $i=0$, this gives
	\[\CH^m X=\Z h^m / A_0 = \Z h^m /\Z (\ind Z) h^m =\Z h^m/\Z n h^m,\]
	while for $i=m-1$, this gives 
	\[
	\Pic X\cong \CH^1 X=\Z h / A_{m-1}=\Z h/\Z [Z]  =\Z h/\Z d h.
	\]
	Thus, $L$ is $d$-torsion. Since $c_1(\calO_{\bbP^m}(1))=h$, we have
	\[
	c_1(L)^m = h^m|_X = h^m +\Z nh^m,
	\]
	which means that $c_1(L)^m$ has order $n$ in $\CH^m X$.
	
	(ii) The scheme $X'$ is affine because the morphism $X'\to X$ is affine
	and $X$ is affine. We continue to use the notation from  the proof of (i). 
	
	By
	\cite[Ex.~2.6.3]{Fulton_1998_intersection_theroy_second_ed},
	for every $i$, there is an exact sequence
	\[
	\CH^{i-1} X\xrightarrow{c_1(\calO(\ell)|_{X})} \CH^i X \xrightarrow{p^*} \CH^i X'\to 0.
	\]
	Since  $c_1(\calO(\ell)|_{X}) = \ell h+ \Z d h$, this allows us
	to   identify
	$\CH^i X'$ with 
	\[\CH^i X/ (\ell h)  \CH^{i-1} X =
	\Z h^i / (A_{m-i} +\Z\ell h^i).\]
	and  $p^*$ with the quotient map. In particular,
	$\Pic X'=\CH^1 X'$ is $\ell$-torsion, hence $L$ is $\ell$-torsion. Furthermore,
	\[
	\CH^mX' = \Z h^m / \Trings{ nh^m,\ell h^m}
	=\Z h^m /\Z\gcd(n,\ell)h^m.
	\]
	Now, by (i),
	\[
	c_1(L')^m=p^*(c_1(L)^m)=p^*(h^m+ \Z n h^m)=h^m+\Z\gcd(n,\ell)h^m  
	\]
	and it follows that $c_1(L')^m$ has order $\gcd(n,\ell)$ in $\CH^m X'$.
\end{proof}

\begin{remark}
	In an earlier version of this text, we gave a different construction
	of $\ell$-torsion  ($\ell>1$) line bundles $L$ with $c_1(L)^m\neq 0$.
	However, we found out that complements of high-index hypersurfaces in projective spaces  give 
	rise to examples of smaller dimension.
	
	In our original construction, we took $X$ to be the total space of the line bundle
	$\calO(-\ell)$ over $\bbP^m$ minus the zero section, and $L$ to be the pullback of
	$\calO(-1)$ along $X\to \bbP^m$. We then replaced $X$ and $L$ with their pullback
	along the standard \emph{Jouanolou device} of $\bbP^m$ \cite{Jouanolou_1972_suite_exacte_de_Mayer_Vietoris}
	to make $X$ affine. 
	The resulting $\ell$-torsion line is induced by the $\umu_\ell$-torsor
	$\uGL_{m+1}/[\begin{smallmatrix}\uGL_m & \\ &  1\end{smallmatrix}] \to 
	\uGL_{m+1}/[\begin{smallmatrix} \uGL_m & \\ & \umu_\ell\end{smallmatrix}]$.
	We omit the details as they will not be needed here; the proof that $L$ is
	$\ell$-torsion and $c_1(L)^m\neq 0$
	is similar to the proof of Proposition~\ref{PR:Gk-bound-I}(ii). (When passing
	to the Jouanolou device, apply  \cite[\href{https://stacks.math.columbia.edu/tag/0GUB}{Tag 0GUB}]{stacks_project}.)
\end{remark}

\begin{example}\label{EX:Gk-and-u-inv}
	Recall that the $u$-invariant of $k$, denoted $u(k)$,
	is the supremum of all the integers $n$ for which $k$ admits an $n$-dimensional
	nonsingular anisotropic quadratic form. We claim that
	\[
	g_k(2,m)=G_k(2,m)=m
	\]
	whenever $1\leq m< u(k)$. In particular, if $u(k)=\infty$, e.g., if $k$ is real,
	then $g_k(2,m)=G_k(2,m)=m$ for all $m$. To see this, let $q:k^{m+1}\to k$
	be an  anisotropic nonsingular quadratic form, let $Z$ be the quadric hypersurface $q=0$ in $\bbP^m$  
	and let $X=\bbP^m-Z$. Since $q$ is nonsingular and anisotropic, $Z$ is integral,
	and
	by Springer's Theorem \cite[II, Thm.~5.3]{Scharlau_1985_quadratic_and_hermitian_forms}, $\ind Z=2$.
	Proposition~\ref{PR:Gk-bound-I}(i)  now
	tells us that the line bundle
	$\calO_{\bbP^m}(1)|_X$ is $2$-torsion and satisfies $c_1(L)^m\neq 0$.
	As a result, $G_k(2,m)\leq m$, and it follows
	that $g_k(2,m)=G_k(2,m)=m$.
	
	In the special case $k=\R$ and $q(x)=\sum_{i=1}^{m+1} x_i^2$,
	the scheme $X$ is isomorphic to the quotient of the $m$-sphere by the antipodal map, 
	and we recover the  example
	in Swan
	\cite[Thm.~4]{Swan_1962_vec_bundles_and_proj_modules}. 
\end{example}

In order to apply Proposition~\ref{PR:Gk-bound-I},
we need to find integral hypersurfaces in $\bbP^m$ having index greater than $1$. Norm forms
of finite field  extensions   give rise to such examples.

\begin{lem}\label{LM:norm}
	Let $K/k$ be a finite separable field extension of degree $n$
	and let $a\in K$ be an element such that $K=k[a]$.
	For every  $m\in\{1,\dots,n-1\}$ set $V_m=\Span\{1,a,\dots,a^{m}\}$,
	and let $Z_m$   be the hypersurface in $\bbP(V_m)\cong \bbP^{m}$ cut
	by the equation  $\Nr_{K/k}(x_0+x_1a+\dots +x_{m}a^{m})=0$. Then $Z_m$ is integral.
	Moreover, if $[K:k]$ is a power
	of
	a prime number $\ell$, then $\ell\mid\ind Z_m$.
\end{lem}

\begin{proof}
	The fact that $Z_m$ is integral follows
	from \cite[Thm.~1]{Flanders_1953_norm_function} and its proof.
	
	Suppose now that $[K:k]=\ell^r$ for a prime number
	$\ell$ and $r\in\N$. We need to show that $\ell\mid \ind Z_m$.
	Let $L$ be a finite $k$-field with $Z_m(L)\neq\emptyset$.
	Then the norm $N_{L\otimes_k K /L}: L\otimes_k K\to L$ has a non-trivial zero.
	This means that   $L\otimes_k K$
	is not a field. 
	This can happen only if $\ell^r=[K:k]$ and $[L:k]$ are not coprime,
	so $\ell$ must divide $[L:k]$.
\end{proof}

\begin{cor}\label{CR:upper-bound-on-Gk}
	Let $\ell$ be a prime number. 
	\begin{enumerate}
		\item[(i)]  If $k$ admits a separable field extension of degree $\ell$,
		then $G_k(\ell,m)=m$ for all $m\in\{1,\dots,\ell -1\}$.
		\item[(ii)] If $k$ admits a separable field extension of degree $\ell^r$ for some
		$r\in\N$, then $G_k(\ell,m)\leq m+1$ for all $m\in \{1,\dots,\ell^r-1\}$.
	\end{enumerate}
\end{cor}

\begin{proof}
	(i)  
	Let $K/k$ be a finite separable field extension of degree $\ell$. 
	Choose $a\in K$ with $K=k[a]$, which exists by the Primitive Element Theorem,
	and define $Z_m $ as in Lemma \ref{LM:norm}.
	Then $Z_m$ is  an integral hypersurface in $\bbP^m$
	having degree and index $\ell$.
	Proposition~\ref{PR:Gk-bound-I}(i) now provides us with a line bundle
	bundle $L$ over $\bbP^m-Z_m$ such that $c_1(L)^m\neq 0$,
	so $G_k(\ell,m)\leq m$.

	(ii) This similar to (i), but we use Proposition~\ref{PR:Gk-bound-I}(ii)
	instead of part (i) of that proposition. 
\end{proof}

Corollary~\ref{CR:upper-bound-on-Gk} cannot be applied to algebraically
closed fields. This is remedied by the following proposition
at the cost of increasing the upper bound on $G_k(\ell,m)$ by $1$.

\begin{prp}\label{PR:Gk-to-Gkx}
	Let $k$ be a field and $\ell,m\in\N\cup\{0\}$ with $\ell\neq 1$.
	Then $G_k(\ell,m)\leq G_{k(t)}(\ell,m)+1$.
\end{prp}

\begin{proof}
	This is shown by \emph{spreading out} from $\Spec k(t)$
	to $\Spec k[t]$.
	What we need is summarized neatly 
	in \cite[Thms.~3.2.1, 6.4.3]{Poonen_2017_rational_points}.
	
	In  detail, let $X$ be a smooth affine $k(t)$-scheme
	of dimension $G_{k(t)}(\ell,m)$ carrying an
	$\ell$-torsion line bundle
	$L$ with $c_1(L)^m\neq 0$.
	Put $S=\Spec k[t]$ and let $\xi=\Spec k(t)$ be the generic point of $S$. 
	By \cite[Thm.~3.2.1]{Poonen_2017_rational_points}, 
	there is a dense open
	subscheme $U$ of $S$ and a $U$-scheme of finite presentation
	$Y$ such that $Y_\xi =X$.
	Moreover, by shrinking $U$  if necessary, we may further assume
	that $Y\to U$ is smooth and $U$ is affine. Thus, $Y$ is a smooth affine $k$-scheme.
	By \cite[Thm.~6.4.3]{Poonen_2017_rational_points} (applied to $G=\uGm$), 
	we can shrink $U$ even further to guarantee
	that there exists  an $\ell$-torsion line bundle $M$ on $Y$ such that the pullback
	of $M$ along $\xi\to S$ is $L$.
	The map $X=Y_\xi\to Y$ 
	induces a ring homomorphism
	$\CH^* Y\to \CH^* X$ which is compatible
	with Chern classes  \cite[\S52F]{Elman_2008_algebraic_geometric_theory_of_quad_forms}.
	Thus, the image of $c_1(M)^m$ in $\CH^* X$ is $c_1(L)^m$, which is nonzero.
	We conclude that $G_k(\ell,m)\leq \dim Y=\dim X+1=G_{k(t)}(\ell,m)+1$.
\end{proof}

\begin{cor}\label{CR:Gk-general-upper-bound}
	Let $k$ be a field and let $m,\ell\in\N\cup\{0\}$ with
	$\ell\neq 1$.
	Then $G_k(\ell,m)\leq m+2$.
	When $\ell$ has a prime divisor greater than $m$ (this holds when $\ell=0$), 
	we moreover have $G_k(\ell,m)\leq m+1$.  
\end{cor}

\begin{proof}
	Let $K=k(t)$. 
	By Proposition~\ref{PR:Gk-to-Gkx},  it is enough
	to show that $G_{K}(\ell,m)\leq m$ when $\ell$ has a prime divisor $p>m$,
	and $G_K(\ell,m)\leq m+1$ in general.  
	Observe that $K$ has a separable field extension of degree $n$
	for every $n\in\N$, for instance, $K[x] / (x^n-tx-t)$.
	Letting $n$ range over the powers of some prime divisor $p$ of $\ell$
	and applying   Corollary~\ref{CR:upper-bound-on-Gk}(ii)
	gives $G_K(\ell,m)\leq G_K(p,m)\leq m+1$.
	If $\ell$ has a prime divisor $p$ greater  than $m$,
	then   applying Corollary~\ref{CR:upper-bound-on-Gk}(i) to that divisor
	gives $G_K(\ell,m)\leq G_K(p,m)=m$.	
\end{proof}

\begin{remark}\label{RM:Paulsen}
	Suppose that $k=\C$. 
	In \cite{Paulsen_2022_degree_of_alg_cycl_on_hypersurf}, Paulsen conjectured the following:
	\begin{enumerate}
		\item[(P)] Let $m,d\in\N$ with $d\geq 2m $ and let $Y$
		be a very general hypersurface of degree $d$ in $\bbP^{m+1}$. Then the degree of every 
		positive-dimensional closed subvariety $Z$ of $Y$ is divisible by $d$.
	\end{enumerate}
	The case $m=3$ is a famous conjecture of Griffiths--Harris.
	We  shall consider   conjecture (P) for a general field  $k$.
	
	Fix $k$ and $m$. If Paulsen's conjecture holds for
	every $d$ sufficiently large, then $G_k(\ell,m)\leq m+1$ for every $ \ell\in \N-\{1\}$.
	Indeed, choose $r\in\N$ coprime to $\ell$ such that (P)
	holds with $d=\ell r$,   let $Y$ be as in the conjecture and let $X=\bbP^{m+1}-Y$. 
	With notation as in the proof of Proposition~\ref{PR:Gk-bound-I},
	the conjecture says that for every $0\leq i< m$, the image of
	the pushforward map    $\CH^i Y\to \CH^{i+1}\bbP^{m+1}=\Z h^{i+1}$ is $ \Z^{i+1} d h^{i+1}$.
	The cokernel of this map is $\CH^{i+1} X$, so it
	can be identified with $\Z h^{i+1}/\Z dh^{i+1}$
	whenever $0\leq i<m$.
	Now put $L=\calO_{\bbP^{m+1}}(r)|_X$. Then $c_1(L)=rh+dh\Z$ while
	$c_1(L)^m=r^m h^m+\Z dh^m $. Since $\ell r=d$ and $\gcd(\ell,r)=1$,
	this means that $L$ is $\ell$-torsion and $c_1(L)^m\neq 0$. As $X$ is affine of dimension $m+1$,
	we conclude that $G_k(\ell,m)\leq m+1$.
	
	When $k=\C$, Paulsen \cite[Thm.~3, Prop.~7]{Paulsen_2022_degree_of_alg_cycl_on_hypersurf} 
	showed that for every $m$ there is a  subset
	$S\subseteq \N$ of positive density such that (P) holds for every $d\in S$.
	The set $S$ consists of numbers that are coprime to $m!$, and together
	with the argument of the last
	paragraph, it can be used to show that $G_\C(\ell,m)\leq m+1$ if
	all the prime factors of $\ell$ are greater than $m$. However, this is already included
	in Corollary~\ref{CR:Gk-general-upper-bound}. 
\end{remark}

The next results give bounds on $G_k(0,m)$.
They rely on deeper results of Bloch, Mohan Kumar, Murthy, Szpiro,
Roitman and Roy
appearing in
\cite{Mohan_1988_cancellation_proj_modules}, 
\cite{Bloch_1989_zero_cycles_num_generators}, 
\cite{Murthy_1994_zero_cycles_and_proj_modules}.

\begin{prp}[{\cite[\S5]{Bloch_1989_zero_cycles_num_generators}}]
	\label{PR:Gk-zero}
	For every $m\in\N$, there is $t\in \N$
	such that for  every field $k$ with $\Char k=0$
	and $\mathrm{trdeg}_\Q k\geq t$,
	we have $G_k(0,m)=m$.
\end{prp}

\begin{proof}
	Suppose first that $k$ is algebraically closed and uncountable.
    We will use several results from \cite{Bloch_1989_zero_cycles_num_generators} that are stated over $k=\C$, but as is evident from their proofs, are actually true over any uncountable algebraically closed field of characteristic $0$.
    Following the beginning of   the proof of 
    \cite[Thm.~5.8]{Bloch_1989_zero_cycles_num_generators},
	let $C_1,\dots,C_m$ be elliptic curves over $k$, put $Y=C_1\times\dots\times C_m$ and let
	$X$ be a dense open affine subset of $Y$.
    Then $\HH^0(Y,\Lambda^m\Omega_Y)\neq 0$.
    By \cite[Lem.~5.2]{Bloch_1989_zero_cycles_num_generators},
    this implies that $\CH^m(X)\neq 0$ (our assumptions on $k$
    are needed here). Moreover, $\CH^m(X)$ is clearly
    generated by products of the form $\alpha_1\cdots \alpha_m$
    with  $\alpha_1,\dots,\alpha_m\in \CH^1(X)$ (because this also holds for $Y$).
    Thus, by 
    \cite[Lem.~5.6]{Bloch_1989_zero_cycles_num_generators}, there
    is a line bundle $L$ over $X$ such that $c_1(L)^m\neq 0$.
    It follows that $G_k(0,m)=m$.

	Suppose now that $k$ is uncountable.
	Let $X'$ 		
	and $L'$ be the $X$ and $L$   constructed in the last paragraph for
	the field $\algcl{k}$. As in the proof of   Proposition~\ref{PR:Gk-to-Gkx},
	there is a finite extension $k_1$ of $k$ such that $X'$
	and $L'$ descend to $k_1$, i.e., there
	is a smooth affine $k_1$-scheme $X $ and a line bundle $L $ over $X$
	such that $ X \times_{k_1} \algcl{k}=X'$
	and $L'$ is the pullback of $L$ along $X'\to X$.
	The image of $c_1(L)^m$
	under $\CH^m X\to \CH^m X'$ is $c_1(L')^m\neq 0$, so we must have
	$c_1(L)^m\neq 0$. As $X$ is a smooth affine $k$-scheme of dimension $m$,
	it follows that $G_k(0,m)=m$.
	
	Finally, suppose that $k$ is countable.
	Let $X'$ and $L'$ be the $X$ and $L$   constructed for the field $\C$.
	Again, there is a finitely generated $\Q$-field $k_0$, 
	such that $X'$ descends to a smooth affine $k_0$-scheme
	$X_0$ and $L'$ descends to a line bundle $L_0$ over $X_0$.
	We take $t:=\mathrm{trdeg}_{\Q} k_0$.
	If  $\mathrm{trdeg}_\Q k\geq t=\mathrm{trdeg}_\Q k_0$, then there
	is a   finite extension $k_1$ of $k$
	that is isomorphic to a field lying between $k_0$ and $\C$.
	Identifying $k_1$ with that field, put $X=X_0\times_{k_0} k_1$
	(viewed as a $k$-scheme) and let $L$ to be the pullback of $L_0$
	along $X\to X_0$. As before,  $c_1(L)^m\neq 0$ (because its image in $\CH^m X'$
	is $c_1(L')^m\neq 0$) and $G_k(0,m)=m$. 
\end{proof}

\begin{remark}\label{RM:positive-char}
	The example in \cite[Thm.~5.8]{Bloch_1989_zero_cycles_num_generators} 
	relies on   a criterion of Roitman  \cite[Lem.~5.2]{Bloch_1989_zero_cycles_num_generators}  
	for showing that the group
	of $0$-cycles on a complete $k$-variety $Y$
	is \emph{infinite dimensional} in the sense that there is no closed subvariety
	$Z\subsetneq Y$ such that $\CH_0(Z)\to \CH_0(Y)$ is surjective.
	This criterion, originally proved for $k=\C$, requires $k$ that is uncountable
	algebraically closed 
	of characteristic $0$.
	
	V.\ Srinivas   pointed out to us that Bloch gave an analogue of Roitman's result for
	uncountable fields of positive characteristic 
	using $\ell$-adic cohomology \cite[Thm.~1A.6, Ex.~1A.7]{Bloch_2010_lectures_alg_cycles} (see also \cite[Prop.~1]{Bloch_1983_corres_alg_cycles}).
	Moreover, it seems likely that with some work, this  result
	could be applied to the example in  \cite[Thm.~5.8]{Bloch_1989_zero_cycles_num_generators}.
	We therefore expect that Proposition~\ref{PR:Gk-zero} also holds in positive characteristic.
\end{remark}

\begin{prp}
	\label{PR:Gk-pos-char}
	Let $m\in\N-\{1\}$ and let $p$ be a prime number.
	Then 
	$G_{\overline{\F}_p}(0,m)=m+1$.
\end{prp}

\begin{proof}
	Since $G_{\overline{\F}_p}(0,m)\leq m+1$ (Corollary~\ref{CR:Gk-general-upper-bound}), 
	it is enough to prove that if $X=\Spec A$ is a smooth affine
	$\overline{\F}_p$-scheme of dimension $m$, then $\CH^m X=0$.
	By 	\cite[Thm.~3.6]{Mohan_1988_cancellation_proj_modules},
	every closed point of $X$ is a complete intersection,
	and 
	by \cite[Cor.~3.4, Thm.~2.14ii)]{Murthy_1994_zero_cycles_and_proj_modules},
	this means that $\CH^m X\cong F^m K_0 (A)=0$.
\end{proof}

\begin{remark}\label{RM:Bloch-Beilinson}
	The characteristic-$0$ counterpart of
	Proposition~\ref{PR:Gk-pos-char}, i.e., whether
	$G_{\algcl{\Q}}(0,m)=m+1$ for all $m\geq 2$,
	is related to the Bloch--Beilinson conjecture.
	Indeed, in some  formulations, e.g.\
	\cite[p.~157]{Krishna_2002_zero_cycles_K_thy_normal_surf} or
	\cite[Cnj.~7.4, Cnj.~7.7]{Srinivas_2008_appl_alg_cycles_to_aff_alg_geo}, the conjecture
	may be stated as follows:
	\begin{enumerate}
		\item[(BB)]
		For every smooth affine $\algcl{\Q}$-variety $X$
		of dimension $m\geq 2$, we have $\CH^m X=0$.
	\end{enumerate}
	If this holds, then we must have $G_{\algcl{\Q}}(0,m)\geq m+1$,
	and then equality holds by Corollary~\ref{CR:Gk-general-upper-bound}.
\end{remark}

We conclude our discussion of $G_k(\ell,m)$ with the following theorem, which implies
Theorem~\ref{TH:bounds-on-generators}.

\begin{thm}\label{TH:Gk-all-bounds}
	Let $k$ be a field and let $\ell,m $ be non-negative integers with $\ell\neq 1$. 
	\begin{enumerate}[label=(\roman*)]
		\item  $G_k(\ell,0)=0$ and $G_k(\ell,1)=1$.
		\item  $G_k(\ell,m) \in\{m,m+1,  m+2\}$.  
		\item  If   $\ell$ is divisible by a prime number greater than
		$m$ (e.g.\ $\ell=0$),
		then   $G_k(\ell,m)\in\{m, m+1\}$.
		\item  If $k$ is algebraically closed, $m>1$ and $\ell>0$, 
		then $G_k(\ell,m)\geq m+1$. Equality holds
		if  $\ell$ is moreover divisible
		by a prime number greater
		than $m$.
		\item  If $\Char k=0$
		and $\mathrm{trdeg}_\Q k$ is infinite,
		then    $G_k(0,m)=m$.
		\item $G_{\quo{\F}_p}(0,m)=m+1$ for every prime $p>0$ and $m\geq 2$.
	\end{enumerate}
\end{thm}

\begin{proof}
	(ii), (iii), (iv), (v), (vi) follow from Corollary~\ref{CR:Gk-general-upper-bound} 
	and Propositions~\ref{PR:Gk-alg-closed-lower-bound}, \ref{PR:Gk-zero}, \ref{PR:Gk-pos-char}.
	As for (i), that $G_k(\ell,0)=0$ is clear --- take $L$ to be the trivial line bundle
	on $\Spec k$.
	To show that $G_k(\ell,1)=1$, note first that
	there is a finite separable $k$-field $k_1$
	and an ordinary elliptic curve $C_1$ over $k_1$.
	(Indeed, all elliptic curves are ordinary in characteristic $0$,
	and if $p=\Char k>0$, then any elliptic curve
	with $j$-invariant not lying in $\F_{p^2}$ is ordinary.)
	Since $C_1$ is ordinary, there is a finite separable
	$k_1$-field $k_2$ and   $P\in C_1(k_2)$ of order
	$\ell$ \cite[\href{https://stacks.math.columbia.edu/tag/03RP}{Tag 03RP}]{stacks_project}.
	Let $C_2=C_1\times_{k_1} k_2$ and let $O$ be the trivial point of $C_2$. 
	Then $X:=C_2-\{O\}$ is  affine and smooth over $k$,
	and $[P]\in \CH^1 X\cong \Pic X$
	is the $1$-st Chern class of a  nontrivial $\ell$-torsion line bundle $L$ over $X$.
\end{proof}

\begin{proof}[Proof of Theorem~\ref{TH:bounds-on-generators}]
	Thanks to Corollary~\ref{CR:gk-leq-Gk} and Lemma~\ref{LM:Forster-bound},
	this follows from  Theorem~\ref{TH:Gk-all-bounds}.
\end{proof}

\section{Main Results}
\label{sec:main-proof}

We finally prove Theorems~\ref{TH:main} and~\ref{TH:main-quantitative}.
We will use the invariant $g_k(\ell,m)$ defined
in Section~\ref{sec:torsion-line-bundles}. 
Recall from Theorem~\ref{TH:bounds-on-generators} that 
$g_k(\ell,m)\in \{m,m+1,m+2\}$ and $g_k(0,m)\in\{m,m+1\}$.

We derive Theorems~\ref{TH:main}
and~\ref{TH:main-quantitative}
from the following theorem.

\begin{thm}\label{TH:explicit-torsor}
	Let $k$ be a field, let $G$ be a closed subgroup of $\uGL_n$ ($n\in\N$),
	let $E \to X$ be a $G$-torsor with $X$
	a quasi-compact $k$-scheme  and let  $V$ be the vector bundle over $X$ corresponding 
	to the $\uGL_n$-torsor $E\times^G \uGL_n\to X$.
	Let $m$ be a non-negative integer such that one of the following holds:
	\begin{enumerate}[label=(\arabic*)]
		\item $m\geq n\spl V$,   or
		\item $m\geq n(\dim X+1)$ and $X$ is quasi-projective over a  field.
	\end{enumerate}
	Suppose further that there is a $k$-field $K$ and $\ell\in\N-\{1\}$
	such that $\umu_{\ell,K}$ embeds in $G_K$.
	Then there is a smooth affine $K$-scheme $X'$ of dimension $g_K(\ell,m)$
	carrying a $G$-torsor
	that is not a specialization of $E\to X$.
\end{thm}

\begin{proof}
	Assumption (2) implies (1) by Proposition~\ref{PR:bound-on-splitting-number}.
	By the definition of $g:=g_K(\ell,m)$,
	there is a $g$-dimensional smooth affine $K$-scheme $X'$ carrying an $\ell$-torsion
	line bundle $L'$ such that $\gen(L')>m$.
	For the sake of contradiction, suppose that $E\to X$ specializes
	to every $G$-torsor over $X'$, i.e., $E\to X$ is weakly versal for the
	class of $k$-schemes $\{X'\}$.
	Then by \cite[Prop.~5.5]{First_2023_highly_versal_torsors},
	the $G_K$-torsor $E_K\to X_K$ is weakly versal for the class of $K$-schemes
	$\{X'\}$. Since $G_K$ contains a copy of $\umu_{\ell,K}$, 
	Theorem~\ref{TH:reduction-to-line-bundles}(i) tells
	us that every $\ell$-torsion line bundle over $X'$   can be generated by $n\spl (V_K)\leq n\spl V$ elements. 
	But that is absurd because
	$L'$ was chosen so that    $\gen(L')>m\geq n\spl V$.
	We therefore conclude that $E\to X$ does not specialize to all $G$-torsors over $X'$.
\end{proof}

We also have the following variant of Theorem~\ref{TH:explicit-torsor} which gives
a slightly better upper bound on $\dim X'$ when $G_K$ contains a copy of $\nGm{K}$.

\begin{thm}\label{TH:explicit-torsor-torus}
	Let $k$, $n$, $G$, $E\to X$, $V$ and $m$ be as in Theorem~\ref{TH:explicit-torsor}.
	Suppose further that there is a $k$-field $K$  
	such that $\nGm{K}$ embeds in $G_K$ and the characters of $\nGm{K}$ occurring
	in the representation $\nGm{K}\to G_K\to\uGL_{n,K}$ include
	the identity character or $x\mapsto x^{-1}$.
	Then there is a smooth affine $K$-scheme $X'$ of dimension $g_K(0,m)$
	carrying a $G$-torsor
	$E'\to X'$ that is not a specialization of $E\to X$.
\end{thm}

\begin{proof}
	This is similar to the proof of Theorem~\ref{TH:explicit-torsor},
	except $L'$ is not necessarily torsion and
	one uses part (ii) of Theorem~\ref{TH:reduction-to-line-bundles}
	instead of part (i).
\end{proof}

We can now prove Theorems~\ref{TH:main}
and~\ref{TH:main-quantitative}  as well as a few   improvements.

\begin{proof}[Proof of Theorem~\ref{TH:main}]
	The implication (b)$\implies$(a) is clear
	and (c)$\implies$(b) is a special case of \cite[Thm.~10.1]{First_2023_highly_versal_torsors}. It remains to show that
	(a)$\implies$(c), that is,
	we need to show that if $X$ is a quasi-compact $k$-scheme
	and $E\to X$ is a $G$-torsor that is weakly  versal for all   finite type regular affine $k$-schemes,
	then $G$ is unipotent.
	
	For the sake of contradiction, suppose that
	$G$ is not unipotent. 
	By  Theorem~\ref{unipiffnomultsubgp}(ii), $G_{\algcl{k}}$ contains a copy
	of $\umu_{\ell,\algcl{k}}$ for some $\ell>1$.
	This means that there is a finite-dimensional $k$-field
	$K$
	such that  $G_{K}$ contains a copy of $\umu_{\ell,K}$.
	Theorem~\ref{TH:explicit-torsor} now gives us
	a $G$-torsor $E'\to X'$, where $X'$ is a smooth affine   $K$-scheme,
	which is not a specialization of $E\to X$. Since $X'$ is regular affine
	and also  of finite type over $k$,
	this contradicts our assumption on $E\to X$. We conclude that $G$ must be unipotent.
\end{proof}

\begin{thm}\label{TH:main-special-case}
	Let $k$ be a field and $G$ be an affine algebraic $k$-group.
	Suppose that $G$ is connected,   or $k$ is perfect.
	Then the   conditions of Theorem~\ref{TH:main}	
	are also equivalent to: 
	\begin{enumerate} 
		\item[(a${}^+$)]
		there exists a $G$-torsor over a quasi-compact
		$k$-scheme that is weakly versal for all \emph{smooth} affine $k$-schemes.
	\end{enumerate}
\end{thm}

\begin{proof}
	As in the proof of Theorem~\ref{TH:main}, it is enough
	to prove that (a${}^+$)  implies that $G$ is   unipotent.
	Suppose that (a${}^+$) holds  but $G$ is not unipotent.
	Then, as in the proof of Theorem~\ref{TH:main},
	given a finite-dimensional $k$-field $K$
	such that $G_{K}$ contains a copy of
	$\umu_{\ell,K}$ with $\ell>1$,
	we could use 
	Theorem~\ref{TH:explicit-torsor} to get
	a smooth affine $K$-scheme $X'$
	and a $G$-torsor $E'\to X'$ that is not a specialization of
	$E\to X$. 
	We will reach the desired contradiction by showing that $K$ can be taken
	to be separable over $k$, and thus $X'$ is smooth over $k$.

	When $k$ is perfect, we could simply take $K$ to be the $k$-field
	from the proof Theorem~\ref{TH:main}.
	If $G$ is connected, then Theorem~\ref{unipiffnomultsubgp}(i)
	tells us that $G$ contains a nontrivial multiplicative group $M$.
	In this case,  take $K$ to be a finite separable $k$-field
	such that $M_{K}$ is diagonalizable.
\end{proof}

Theorem~\ref{TH:main-quantitative} is part (i) of  the following theorem.

\begin{thm}\label{TH:main-quantitative-speical-cases}
	Let $G$ be a non-unipotent algebraic  subgroup of the algebraic $k$-group $\uGL_n$,
	let $X$ be a $d$-dimensional $k$-scheme that is quasi-projective over a field 
	and let $E\to X$ be a $G$-torsor. Then:
	\begin{enumerate}[label=(\roman*)]
		\item $E\to X$ is not weakly versal for the class of
		finite type regular affine $k$-schemes of dimension $\leq  n(d+1)+2$.
		\item If $k$ is perfect or $G$ is connected,
		then $E\to X$ is not weakly versal for the class of  smooth  affine $k$-schemes
		of dimension $\leq n(d+1)+2$.
		\item If $G$ contains a nontrivial torus,
		then $E\to X$ is not weakly versal for the class of {smooth} affine $k$-schemes
		of dimension $\leq n(d+1)+1$.
		\item If 
		$\Char k=0$,
		$\mathrm{trdeg}_\Q k$ is infinite
		and $G_{\algcl{k}}$ contains a copy of $\nGm{\algcl{k}}$ such that
		characters of the representation $\nGm{\algcl{k}}\to G_{\algcl{k}}\to \uGL_{n,\algcl{k}}$
		include the identity character or $x\mapsto x^{-1}$,
		then $E\to X$ is not weakly versal for the class of {smooth} affine $k$-schemes
		of dimension $\leq n(d+1)$.
	\end{enumerate}
\end{thm}

\begin{proof}
	(i) 
	By Theorem~\ref{unipiffnomultsubgp}(ii), there is a
	a finite dimensional $k$-field $K$ such that $G_{K}$
	contains a copy of $\umu_{\ell,K}$ for some $\ell>1$.
	Put $m=n(d+1)$ and let $E'\to X'$ be the $G$-torsor guaranteed in Theorem~\ref{TH:explicit-torsor}.
	Then $E\to X$ does not specialize to $E'\to X'$.
	Since $\dim X'=g_K(\ell,m)\leq m+2$, we conclude that $E\to X$ is not weakly
	versal for   finite type regular affine $k$-schemes of dimension $\leq n(d+1)+2$.

	(ii) As in the proof of Theorem~\ref{TH:main-special-case},
	there is a separable finite-dimensional $k$-field $K$
	such that $G_{K}$ contains a copy of $\umu_{\ell,K}$
	for some $\ell>1$. Proceed as in (i)
	using this $K$. The scheme $X'$ is smooth over $k$ because it is smooth
	over $K$ and $K$ is smooth over $k$.
	
	(iii) Let $T$ be a nontrivial subtorus of $G$
	and let $K$ be a finite separable $k$-field splitting $T$.
	Then $G_K$ contains a copy $\nGm{K}$. Let $\ell$ be a prime number
	greater than $m:=n(d+1)$. Then $G_K$
	contains a copy of 
	of $\umu_{\ell,K}$. Proceed as in (i), but note that $g_K(\ell,m)\leq m+1$
	by Theorem~\ref{TH:bounds-on-generators}(iii).
	
	(iv) The assumption on $G$
	implies that there is a finite separable $k$-field $K$ such that 
	$G_K$ contains a copy of $\nGm{K}$
	and the characters of the representation $\nGm{K}\to G_K\to \uGL_{n,K}$
	include $x\mapsto x$ or $x\mapsto x^{-1}$. 
	We now argue as in (i) using Theorem~\ref{TH:explicit-torsor-torus}
	instead of Theorem~\ref{TH:explicit-torsor}. Theorem~\ref{TH:bounds-on-generators}(v)
	tells us that $g_K(0,m)=m$.
\end{proof}

An immediate application of Theorem~\ref{TH:main-quantitative-speical-cases} is the following.

\begin{cor}\label{CR:there-are-nontriv-torsors}
	Let $G$ be a non-unipotent algebraic subgroup of the algebraic $k$-group $\uGL_n$.
	Then  there exists  a finite type regular affine $k$-scheme 
	of dimension $n+2$
	carrying 
	a non-trivial $G$-torsor $E\to X$. If $k$ is perfect or $G$ is connected,
	then $X$ can be taken to be smooth as well.
\end{cor}

\begin{proof}
	Apply Theorem~\ref{TH:main-quantitative-speical-cases} to the $G$-torsor $G\to \Spec k$.
\end{proof}

As we noted in the introduction, this result can be significantly improved
in terms of dimension when $G$ contains a nontrivial torus; this will be shown
in the next section.
Another way to strengthen Corollary~\ref{CR:there-are-nontriv-torsors}
is the following.

\begin{cor}\label{CR:nontrivial-torsor-after-ext}
	Let $H$ be an algebraic subgroup of the algebraic $k$-group $\uGL_n$  
	and let $G$ be a non-unipotent  algebraic subgroup of $H$.
	Then there exists a $G$-torsor $E\to X$, with $X$ a finite type regular affine $k$-scheme 
	such that the $H$-torsor $E\times^GH\to X$ is   nontrivial.
	Furthermore, $X$ can be taken so that $\dim X\leq  n(\dim H-\dim G+1)+2$.
\end{cor}

The corollary is obtained by applying Theorem~\ref{TH:main-quantitative-speical-cases} to
the $G$-torsor $H\to H/G$, so improvements analogous to parts (ii), (iii), (iv) of Theorem~\ref{TH:main-quantitative-speical-cases} hold. We leave it to the reader to work them out explicitly.

\begin{proof}
	Let $X$ be a $k$-scheme. It is well-known that there
	is a short exact sequence of pointed sets
	\[
	(H/G)(X)\xrightarrow{\alpha} \mathrm{H}^1_{\fppf}(X,G)\xrightarrow{\beta}\mathrm{H}^1_{\fppf}(X,H)
	\]
	in which $\alpha$ takes an $X$-point $p:X\to H/G$
	to the cohomology class of the $G$-torsor obtained by base-changing
	$H\to H/G$ along $p$, and $\beta$ takes the cohomology class of a $G$-torsor
	$E\to X$ to the cohomology class of $E\times^GH\to X$.
	From the sequence we see that every $G$-torsor $E\to X$
	such that $E\times^G H\to X$ is trivial is a specialization of the $G$-torsor
	$H\to H/G$. Thus, for any $G$-torsor $E\to X$ that is a not a specialization
	of $H\to H/G$, the extension $E\times^G H\to X$ is a nontrivial $H$-torsor.
	With this observation at hand, the corollary follows by applying Theorem~\ref{TH:main-quantitative-speical-cases} to the $G$-torsor $H\to H/G$. (Note that $H/G$ is quasi-projective.)
\end{proof}

\section{Non-Trivial Torsors}
\label{sec:nontriv}

We finish with using results from Sections~\ref{sec:reduction} 
and~\ref{sec:torsion-line-bundles}  to prove
Theorem~\ref{TH:non-triv-torsors-main}.
This will improve
Corollary~\ref{CR:nontrivial-torsor-after-ext} for algebraic groups containing a nontrivial torus. 

We will derive Theorem~\ref{TH:non-triv-torsors-main}
from a slightly more general result. To phrase it, we introduce a variation  of the invariant $G_k(\ell,m)$
considered in Section~\ref{sec:torsion-line-bundles}. For every $m\in \N\cup\{0\}$,
let
\[
\tilde{G}_k(m)
\]
denote the smallest   $d\in\N\cup\{0\}$ such that for every $n\in\N$, there is a smooth affine
$d$-dimensional $k$-scheme and a line bundle $L$ such that $n \cdot c_1(L)^m\neq 0$
in $\CH^m X$.

\begin{thm}\label{TH:there-are-better-nontriv-torsors}
	Let $G$ be an affine algebraic group over $k$ containing a nontrivial torus (resp.\ central torus).
	Then there exists a smooth affine $k$-scheme of dimension at most $ \tilde{G}_k(2)$
	(resp.\ $ \tilde{G}_k(1)$)
	carrying a non-trivial $G$-torsor.
\end{thm}

\begin{proof}[Proof of Theorem~\ref{TH:non-triv-torsors-main} using Theorem~\ref{TH:there-are-better-nontriv-torsors}]
	It is enough to show that $\tilde{G}_k(m)\leq m+1$, $\tilde{G}_k(1)\leq 1$ and,
	provided that $\Char k=0$ and $\mathrm{trdeg}_{\Q} k$ is infinite,
	$\tilde{G}_k(m)\leq m$. 
	
	For the first inequality, given $n\in\N$, choose some prime number $\ell$
	larger than $m$ and $n$. By Theorem~\ref{TH:Gk-all-bounds}(iii),
	there is a smooth affine $k$-scheme $X$ with $\dim X\leq m+1$
	and an $\ell$-torsion line bundle $L$ such that $c_1(L)^m\neq 0$.
	Since $c_1(L)$ is $\ell$-torsion and $\ell\nmid n$, we must have $n\cdot c_1(L)^m\neq 0$.
	Thus, $\tilde{G}_k(m)\leq m+1$.
	
	The second inequality is shown similarly, using Theorem~\ref{TH:Gk-all-bounds}(i).
	
	For the last inequality, write $K=\algcl{k}$. We may assume that $m>1$ as we already
	showed that $\tilde{G}_k(1)\leq 1$.
	By Theorem~\ref{TH:Gk-all-bounds}(v), there is a smooth affine $K$-scheme
	$X$ with $\dim X=m$ and a line bundle $L$ such that $c_1(L)^m\neq 0$.
	By a result of Murthy    \cite[Thms.~2.11, 2.14]{Murthy_1994_zero_cycles_and_proj_modules}, 
	$\CH^m X$ is torsion-free (here we need $m>1$ and $K$ to be algebraically closed).
	Thus, $n\cdot c_1(L)^m\neq 0$ for all $n\in\N$. The $K$-scheme $X$ and the line bundle $L$
	descend to some finite-dimensional $k$-field, which is separable over $k$ because
	$\Char k=0$, hence $\tilde{G}_k(m)\leq m$.
\end{proof}

We turn to prove Theorem~\ref{TH:there-are-better-nontriv-torsors}.
We first prove the following lemma.

\begin{lem}\label{LM:sym-func}
	Let $\sigma_1,\sigma_2:\Z^n\to \Z$ be the first and second elementary symmetric functions.
	Then for every $x\in \Z^n-\{0\}$, either $\sigma_1(x)\neq 0$  or $\sigma_2(x)\neq 0$.
\end{lem}

\begin{proof}
	Write $x=(x_1,\dots,x_n)$.
	We have
	$\sigma_1(x)^2-2\sigma_2(x) = \sigma_1(x_1^2,\dots,x_n^2) \neq   0  $,
	so at least one of $\sigma_1(x)$, $\sigma_2(x) $ must be nonzero.
\end{proof}

\begin{proof}[Proof of Theorem~\ref{TH:there-are-better-nontriv-torsors}]
	Let $T$ be a nontrivial torus contained in $G$ (resp.\ the center of $G$).
	It is harmless to base-change from $k$ to a finite separable extension splitting $T$.
	We may therefore assume that $G$ (resp.\ its center) contains a copy of $\uGm$.
	
	Choose a homomorphism $G\to\uGL_n$ such that the composition $\vphi:\uGm\to G\to \uGL_n$
	is nontrivial, e.g., a faithful representation of $G$.
	It is enough to show that there is a smooth affine $k$-scheme
	$X$ with $\dim X\leq \tilde{G}_k(2)$ (resp.\ 
	$\dim X\leq \tilde{G}_k(1)$) 
	and a $\uGm$-torsor $E$ such that $E\times^{\uGm}\uGL_n$ is a nontrivial $\uGL_n$-torsor.
	Indeed, since $E\times^{\uGm}\uGL_n = (E\times^{\uGm} G)\times^G \uGL_n$,
	this would force $E\times^{\uGm} G\to X$ to be a nontrivial $G$-torsor.
	
	Let $\{x\mapsto x^{a_i}\}_{i=1}^n$ be the characters of the representation
	$\vphi:\uGm\to \uGL_n$ (with multiplicities). We observed
	in the proof of Lemma~\ref{LM:pre-reduction}
	that if $L$ is the line bundle corresponding to $E$,
	then the vector bundle  corresponding to $E\times^{\uGm}\uGL_n$
	is $V:=L^{\otimes a_1}\oplus\dots\oplus L^{\otimes a_n}$.
	Thus,   the total
	Chern class of $V$ is $c(V)=c (L^{\otimes a_1})\cdots c(L^{\otimes a_n})=\prod_{i=1}^n (1+a_i c_1(L))$.
	Writing $a=(a_1,\dots,a_n)\in\Z^n$, this means that
	\[
	c_i(V)=\sigma_i(a)c_1(L)^i
	\]
	for all $i\in\{1,\dots,n\}$, where $\sigma_i$ is the $i$-th elementary symmetric function.
	Since $a\neq 0$, Lemma~\ref{LM:sym-func} tells us that $\sigma_m(a)\neq 0$
	for some $m\in\{1,2\}$. 
	
	We now choose  $X$ to be a $\tilde{G}_k(m)$-dimensional
	smooth affine $k$-scheme carrying a line bundle $L$ with $\sigma_m(a)\cdot c_1(L)^m\neq 0$.
	By what we just showed, 
	$c_m(V)=\sigma_m(a)\cdot c_1(L)^m \neq 0$ and it follows 
	that $E\times^{\uGm} \uGL_n\to X$ is a nontrivial $\uGL_n$-torsor.
	As $\tilde{G}_k(m)\leq \tilde{G}_k(2)$ (because $m\in\{1,2\}$), this proves the theorem
	when $T$ is not assumed to be central in $G$.
	
	When $T$ is contained in the center of $G$,
	we choose $\vphi:G\to \uGL_n$   in such a way that $\uGm$ is also mapped  
	into the center of $\uGL_n$. (Such   representations exist, e.g., start with a faithful
	representation $V$ of $G$, chose a nontrivial character $\chi:\uGm\to \uGm$
	occurring in $\uGm\to G\to\uGL(V)$  and take the eigenspace   $V_\chi$.)
	Then there is $a\in \Z-\{0\}$ such that $a_1=\dots=a_n=a$.
	This means that $\sigma_1(a)=na\neq 0$,   so we can take $m=1$ and get 
	$\dim X\leq \tilde{G}_k(1)$.
\end{proof}

We finish with showing that there 
are algebraic groups for which Theorem~\ref{TH:non-triv-torsors-main}
is the best possible result in terms of dimension.

\begin{example}\label{EX:nontrivial-torsor-lower-bound}
	Consider the algebraic $k$-group $\uSL_n$ ($n>1$). By 
	Theorem~\ref{TH:non-triv-torsors-main},   there is a smooth
	affine $k$-scheme $X$ with $\dim  X\leq 3$
	carrying a nontrivial $\uSL_n$-torsor. Provided that $\Char k=0$
	and $\mathrm{trdeg}_{\Q}k$ is infinite, we can even take $X$ with $\dim X=2$.

	On the other hand, every $\uSL_n$-torsor over an affine noetherian scheme
	of dimension smaller than $2$ is trivial.
	Indeed, it is well-known that there is a bijection between
	isomorphism classes of
	$\uSL_n$-torsors over $X$ and rank-$n$ vector bundles   $X$ 
	with trivial determinant,
	in which the trivial torsor corresponds to $\calO_X^n$.
	(This can be seen, for instance, by inspecting the long fppf cohomology
	exact sequence associated to the short exact sequence
	$\uSL_{n}\hookrightarrow \uGL_{n}\xonto{\det}\uGm$.)
	Let $E$ be an $\uSL_n$-torsor over    a   noetherian affine
	$k$-scheme $X$ with $\dim X<2$  and let $V$ be the corresponding
	vector bundle with $\det(V) \cong \calO_X$.
	By a theorem of Bass--Serre \cite[Thm.~8.2]{Bass_1964_K_theory_and_stable_algs}
	(see also \cite[Thm.~1]{Serre_1957_modules_projectifs}), our assumptions on $X$
	imply that   $V\cong \calO^{n-1}_X\oplus L$
	for some line bundle $L$, so $\det(V)\cong L$.
	Since we also have $\det(V)\cong\calO_X$,
	it  follows that  $V\cong \calO_X^n$  and  $E\to X$ 
	is trivial.
\end{example}

For algebraically closed fields, we can make Example~\ref{EX:nontrivial-torsor-lower-bound}
more precise  and even extend it to semisimple groups of type $G_2$. The latter
case is essentially due to Asok, Hoyois and Wendt \cite{Asok_2019_octonion_algebras_A1_homotopy}.

\begin{prp}\label{PR:triviality-of-tors-dim-2}
	Let $k$ be an algebraically closed field, let $X$ be a smooth irreducible affine
	$k$-scheme with $\dim X<3$, let $n>1$ and let
	$G$ denote the semisimple algebraic $k$-group of type $G_2$. Then the following are equivalent:
	\begin{enumerate}[label=(\alph*)]
		\item $\CH^2 X=0$;
		\item every $\uSL_n$-torsor over $X$ is trivial;
		\item every $G$-torsor over $X$ is trivial.
	\end{enumerate}
\end{prp}

\begin{proof}
	We start with showing that (a)$\iff$(c).
	The field $k(X)$ has cohomological dimension at most $2$,
	and thus every $G$-torsor over $X$ is trivial over the generic point 
	\cite[Apx.~2.3.3]{Serre_2002_Galois_cohomology_english}.	
	Thus, by \cite[Thm.~1]{Asok_2019_octonion_algebras_A1_homotopy},
	$\CH^2 X$ is in bijection with $\HH^1_{\fppf}(X,G)$  and the equivalence is immediate.
	
	We proceed with showing that (a)$\iff$(b).
	Suppose first that (a) holds.
	Let $E$ be an $\uSL_n$-torsor over $X$ and let $V$ be its corresponding
	vector bundle  as in Example~\ref{EX:nontrivial-torsor-lower-bound}.
	We have seen
	that $E\to X$ is trivial when $\dim X<2$, so assume $\dim X=2$.
	By the Bass--Serre Theorem \cite[Thm.~1]{Serre_1957_modules_projectifs},
	$V\cong V'\oplus \calO_X^{n-2}$ for some rank-$2$ vector bundle $V'$.
	Since $\CH^2 X=0$, we have $c_2(V')=0$.
	By a theorem of Murthy \cite[Thm.~3.8]{Murthy_1994_zero_cycles_and_proj_modules},
	this means that there is a line bundle $L$ such that $V'\cong L\oplus \calO_X$
	and as in Example~\ref{EX:nontrivial-torsor-lower-bound}, we get that $E\to X$ is trivial.
	
	Suppose now that (b) holds.
	If $\dim X<2$, then $\CH^2 X=0$, so
	assume  $\dim X=2$. Let $c\in \CH^2 X$.
	As explained in \cite[\S1, a)]{Kumar_1982_alg_cycles_over_affine_3folds},
	there is a rank-$2$ vector bundle $V'$ with $c_1(V')=0$ and $c_2(V')=c$.
	Since $c_1(V')=c_1(\det(V'))=0$, it follows that $\det(V')\cong \calO_X$.
	Let $E\to X$ be the $\uSL_n$-torsor corresponding to $V:=V'\oplus \calO_X^{n-2}$.
	By (c), $E$ is trivial, so $V$ is free,
	and it follows that $c=c_2(V')=c_2(V)= 0$.
\end{proof}

\begin{cor}\label{CR:triviality-of-tors-dim-2}
	Let $p$ be a prime number and let
	$G$ be a semisimple simply connected algebraic  $\quo{\F}_p$-group of
	type $A_n$ or $G_2$.
	Then every $G$-torsor over a smooth affine $\quo{\F}_p$-scheme
	$X$ of dimension $<3$ is trivial.
\end{cor}

\begin{proof}
	We observed in the proof of Proposition~\ref{PR:Gk-pos-char}
	that $\CH^2 X=0$, so this follows from Proposition~\ref{PR:triviality-of-tors-dim-2}.
\end{proof}

\begin{remark}\label{RM:triviality-of-tors-dim-2}
	Let $G$ be a semisimple simply connected  algebraic group over $\algcl{\Q}$ of
	type $A_n$ or $G_2$. By Proposition~\ref{PR:triviality-of-tors-dim-2},
	the statement that every  $G$-torsor  over a smooth affine $2$-dimensional
	$\algcl{\Q}$-scheme is trivial 
	is equivalent to 
	the Bloch--Beilinson conjecture  for surfaces  \cite[p.~157, (b)]{Krishna_2002_zero_cycles_K_thy_normal_surf}.
\end{remark}

\bibliographystyle{plain}
\bibliography{bib_versal}

\end{document}